\documentclass[preprint,12pt]{elsarticle}

\usepackage{graphicx}%
\usepackage{multirow}%
\usepackage{amsthm}%
\usepackage[title]{appendix}%
\usepackage{xcolor}%
\usepackage{textcomp}%
\usepackage{manyfoot}%
\usepackage{booktabs}%
\usepackage{listings}%
\usepackage{amsmath,dsfont,amssymb,amsfonts,mathrsfs}
\usepackage[T1]{fontenc} 
\usepackage{subfigure,subcaption}
\usepackage[ruled,vlined]{algorithm2e}
\usepackage{algpseudocode}
\usepackage{adjustbox}
\newcounter{Theorem}
\newtheorem{theorem}[Theorem]{Theorem}
\newtheorem{proposition}[Theorem]{Proposition}
\newtheorem{lemma}[Theorem]{Lemma}

\newtheorem{remark}[Theorem]{Remark}

\newcommand{\la}{\langle}
\newcommand{\ra}{\rangle}

\date{}

\begin{document}
	
	\begin{frontmatter}
		
		
		
		\title{A Modified Dai-Liao Spectral Conjugate Gradient Method with an  Application to Signal Processing}
		

		\author{D. R. Sahu} 
		\ead{drsahudr@gmail.com}
		
		\affiliation{organization={Department of Mathematics, Banaras Hindu University},
			city={Varanasi},
			country={India}}
		
		\author{Shikher Sharma} 
		\corref{Shikher Sharma}
		\ead{shikhers043@gmail.com}
		
		\affiliation{organization={Department of Mathematics, The Technion -- Israel Institute of
				Technology},
			city={Haifa},
			country={Israel}}
		
		\author{Pankaj Gautam} 
		\ead{pgautam908@gmail.com}
		
		\affiliation{organization={Department of Applied Mathematics and Scientific Computing, Indian Institute of Technology Roorkee},
			country={India}} 
			
				\author{Simeon Reich} 
			\ead{sreich@technion.ac.il}
			
			\affiliation{organization={Department of Mathematics, The Technion- Israel Institute of
					Technology},
				city={Haifa},
				country={Israel}}
				
		\begin{abstract}
			We propose and study a variant of the Dai-Liao spectral conjugate gradient method, developed through an analysis of eigenvalues and inspired by a modified secant condition. We show that our proposed method is globally convergent for general nonlinear functions under standard assumptions. By incorporating the new secant condition and a quasi-Newton direction, we introduce updated spectral parameters. These changes ensure that the resulting search direction satisfies the sufficient descent property without relying on any line search.  Numerical experiments show that the proposed algorithm performs better than several existing methods in terms of convergence speed and computational efficiency. Its effectiveness is further demonstrated through an application to signal processing.
		\end{abstract}
		
		%
		
		
	\end{frontmatter}
	
		
		

		\section{Introduction}

		Let $\mathbb{R}^m$ denote the Euclidean space equipped with the usual  inner product $\langle \cdot , \cdot \rangle$ and its associated norm $\|\cdot\|$.  Let $f: \mathbb{R}^m \to \mathbb{R}$. Our goal is to solve the following minimization problem:
		
		\begin{equation}\label{CGeq}
			\text{find } x^* \in \mathbb{R}^m \text{ such that } f(x^*) = \min_{x \in \mathbb{R}^m} f(x).
		\end{equation}
		We denote the set of minimizers of $f$ by $\Omega_{\text{min}}(f,\mathbb{R}^m)$.
		
		A common approach to find a solution to \eqref{CGeq} is the following iterative method:  Let $x_0\in \mathbb{R}^m$ be an initial point. Given the current iterate $x_n$, the next iterate $x_{n+1}$ is computed by
		\begin{equation}\label{SearchDM}
			x_{n+1}=x_n+\alpha_nd_n  \text{ for all } n \in \mathbb{N}_0=\mathbb{N}\cup \{0\},
		\end{equation}
		where $\alpha_n>0$ is the stepsize determined by exact or inexact line search and $d_n$ is a suitable search direction. Methods that follow the update rule \eqref{SearchDM} are referred to as line search methods. These methods require only the search direction $d_n\in \mathbb{R}^m$ and the stepsize $\alpha_n\in (0,\infty)$.
		An appropriate descent direction $d_n$ must be chosen so that it satisfies the descent condition:
		\begin{equation}\label{DescentCond}
			\langle \nabla f(x_n), d_n\rangle < 0  \text{ for all } n \in \mathbb{N}_0.
		\end{equation}
		
		A common choice for a descent direction is $d_n = -\nabla f(x_n)$, which simplifies the general line search method \eqref{SearchDM} to the familiar gradient descent (GD) method: 
		$$	x_{n+1} = x_n - \alpha_n \nabla f(x_n) \text{ for all } n \in \mathbb{N}_0.$$
		If there exists a constant $c > 0$ such that
		\begin{equation*}
			\langle \nabla f(x_n), d_n\rangle \leq -c\|\nabla f(x_n)\|^2 \text{ for all } n \in \mathbb{N}_0,
		\end{equation*}
		then the vector $d_n$ is said to satisfy the sufficient descent condition.
		
		Conjugate gradient (CG) methods are among the most effective algorithms for solving large-scale unconstrained optimization problems due to their low storage requirements. These methods require only first-order derivatives, makes them highly suitable for large-scale applications. 
		The conjugate gradient (CG) method was originally proposed by Hestenes and Stiefel \cite{Hestenes1952} for the solution of linear systems. Later, Fletcher and Reeves \cite{Fletcher} extended this concept and formulated the first nonlinear CG method for unconstrained optimization. Due to their strong theoretical foundation and excellent numerical performance, many other versions of the CG method have been proposed (see \cite{Polak1969, Dai2001, Fletcher1987, Dai1999} and the references therein).


In general, CG methods use the iterative scheme 
\eqref{SearchDM}, where the search direction $d_n$ is defined by
		\begin{equation}\label{Search_direction}
			d_n=\begin{cases}
				-\nabla f(x_0), & \text{ if }  n=0,\\
				-\nabla f(x_{n}) + \beta_{n-1} d_{n-1}, & \text{ if } n \geq 1,
			\end{cases}
		\end{equation}
		where $\beta_n$ is the CG parameter. There are many formulas for $\beta_n$, each resulting in different computational behavior \cite{HagerZhang2006}.
		
		One of the most well-known CG methods was proposed by Dai and Liao \cite{Dai2001}. Their approach is based on an extended conjugacy condition, where the CG parameter is defined by
		\begin{equation}\label{DLmethod}
			\beta_n^{DL} = \frac{\langle \nabla f(x_{n+1}), y_n \rangle}{\langle d_n, y_n \rangle} - t \frac{\langle \nabla f(x_{n+1}), s_n \rangle}{\langle d_n, y_n \rangle} \text{ for all } n\in \mathbb{N}_0,
		\end{equation}
		where $t\in (0,\infty)$, $s_n = x_{n+1} - x_n$, and $y_n = \nabla f(x_{n+1}) - \nabla f(x_n)$. Note that when $t = 0$, $\beta_n^{DL}$ simplifies to the CG parameter proposed by Hestenes and Stiefel \cite{Hestenes1952}:
		\[	\beta_n^{HS} = \frac{\langle \nabla f(x_{n+1}), y_n \rangle}{\langle d_n, y_n \rangle} \text{ for all } n\in \mathbb{N}_0.\]
		The CG parameter proposed by Hager and Zhang \cite{HagerZhang2005} is defined by
		\[	\beta_n^{HZ} = \frac{\langle \nabla f(x_{n+1}), y_n \rangle}{\langle d_n, y_n \rangle} - 2 \frac{\| y_n \|^2}{\langle d_n, y_n \rangle} \frac{\langle \nabla f(x_{n+1}), d_n \rangle}{\langle d_n, y_n \rangle} \text{ for all } n\in \mathbb{N}_0.\]
	  The parameter $\beta_n^{HZ}$ can be viewed as an adaptive version of \eqref{DLmethod} for the specific choice $t = 2 \frac{\| y_n \|^2}{\langle d_n, y_n \rangle}$.
		Similarly, the CG parameter proposed by Dai and Kou \cite{DaiKou2013} is defined by
		\[	\beta_n^{DK}(\tau_n) = \frac{\langle \nabla f(x_{n+1}), y_n \rangle}{\langle d_n, y_n \rangle} - \left( \tau_n + \frac{\| y_n \|^2}{\langle s_n, y_n \rangle} - \frac{\langle s_n, y_n \rangle}{\| s_n \|^2} \right) \frac{\langle \nabla f(x_{n+1}), s_n \rangle}{\langle d_n, y_n \rangle},\]
		where $\tau_n$
		is a parameter that plays the role of a scaling factor in the scaled memoryless Broyden–Fletcher–Goldfarb–Shanno (BFGS) method, which can be viewed as another adaptive variant of

		Motivated by Shanno’s matrix interpretation of CG methods \cite{Shanno1978}, Babaie-Kafaki and Ghanbari \cite{KafakiGh2014} showed that the Dai–Liao search directions can be written in the form $d_{n+1} = -P_{n+1} \nabla f(x_{n+1})$ for all $n \in \mathbb{N}_0$, where $P_{n+1} = I - \frac{s_n y_n^T}{\langle s_n, y_n \rangle} + t \frac{s_n s_n^T}{\langle s_n, y_n \rangle}$
		is a nonsingular matrix for $t \in (0,\infty)$ and $\langle s_n, y_n \rangle \neq 0$; see also \cite{KafakiEJOR2014}. By performing an eigenvalue analysis on the symmetrized version of $P_{n+1}$, defined by $A_{n+1} = \frac{P_{n+1} + P_{n+1}^T}{2}$,
		we can derive a two-parameter formula for $t$ in the Dai-Liao CG parameter \eqref{DLmethod}, given by
		\begin{equation}\label{CGpara}
			t_n^{p,q} = p \frac{\| y_n \|^2}{\langle s_n, y_n \rangle} - q \frac{\langle s_n, y_n \rangle}{\| s_n \|^2} \text{ for all } n\in \mathbb{N}_0,
		\end{equation}
		which ensures the descent condition \eqref{DescentCond} for $p \in  \left(\frac{1}{4},\infty\right)$ and $q \in \left(-\infty, \frac{1}{4}\right)$. However, in cases where $\frac{\| y_n \|^2}{\langle s_n, y_n \rangle}$ and $\frac{\langle s_n, y_n \rangle}{\| s_n \|^2}$ are nearly equal, their subtraction may not be numerically stable \cite{Kafaki2016}. Moreover, the parameter $t$ in the CG parameter $\beta_n^{DL}$ must be nonnegative. To address these issues, they restricted $q$ in \eqref{CGpara} to be nonpositive and proposed the following CG parameter:
		
		\[	\beta_n^{DDL} = \frac{\langle \nabla f(x_{n+1}), y_n \rangle}{\langle d_n, y_n \rangle} - t_n^{p,q} \frac{\langle \nabla f(x_{n+1}), s_n \rangle}{\langle d_n, y_n \rangle} \text{ for all } n\in \mathbb{N}_0, \]
		where $t_n^{p,q}$ is defined by \eqref{CGpara} with $p \in  \left( \frac{1}{4},\infty\right)$ and $q \in (-\infty,0]$. Note that for the parameter choice $(p, q) = (2, 0)$ in \eqref{CGpara}, $\beta_n^{DDL}$ reduces to $\beta_n^{HZ}$. Similarly, if $(p, q) = (1, 0)$, then $\beta_n^{DDL}$ is equivalent to $\beta_n^{DK}(\tau_n)$
		with the optimal choice of $\tau_n = \frac{\langle s_n, y_n \rangle}{\| s_n \|^2}$. For additional variants of the Dial–Liao conjugate gradient method, we refer to \cite{Kafaki2025,Kafaki2025Opt}.

		Another significant category of iterative methods used for solving unconstrained optimization problems is the spectral gradient method. Initially introduced by Barzilai and Borwein \cite{Barzilai1988}, these methods were later advanced by Raydan \cite{Raydon1997}, who developed a spectral gradient approach specifically for large-scale unconstrained optimization challenges. The spectral conjugate direction is defined by 
		\[	d_{n}=	\begin{cases}
			-\nabla f(x_0), & \text{ if } n=0, \\
			-\theta_{n} \nabla f(x_{n}) + \beta_{n-1}d_{n-1}, & \text{ if } n\geq 1,
		\end{cases}\]
		where $\{\theta_{n}\}\subset (0,\infty)$ is the spectral parameter and $\{\beta_n\}\subset (0,\infty)$ is the conjugate parameter.
		
		Building on the encouraging numerical results associated with spectral gradient methods \cite{Barzilai1988, Raydon1997}, Birgin and Mart\'inez \cite{Birgin2001} proposed the integration of spectral gradient techniques with conjugate gradient methods, culminating in the development of the first spectral conjugate gradient (CG) method. The spectral CG methods represent a powerful enhancement of traditional CG methods by incorporating the principles of spectral gradients within the conjugate gradient framework. Continuing in this direction, Yu and co-authors \cite{Yu2008OMS} adapted Perry’s spectral CG formulation \cite{Perry1978}, obtaining a more effective variant of the method. 
Inspired by the good convergence of quasi-Newton techniques, Andrei \cite{Andrei2007} proposed a CG method using a scaled BFGS preconditioner, and this idea was further developed in later works \cite{Andrei2008, Andrei2010EJOR, Babaie2013MMA}.
		
		Li and Fukushima \cite{Li2001} proposed a modified form of the secant condition, and Zhou and Zhang \cite{Zhou2006} later built on their idea to obtain the following modified secant relation:
		\begin{equation}\label{MSCeq1}
			\nabla^2 f(x_{n+1}) s_n = z_n,
		\end{equation}
		where
		\begin{equation}\label{MSCeq2}
			z_n = y_n + h_n \|\nabla f(x_n)\|^r s_n, \quad h_n = \nu + \max \left\{-\frac{\langle s_n, y_n \rangle}{\|s_n\|^2}, 0\right\} \|\nabla f(x_n)\|^{-r}
		\end{equation}
		with $r,\nu\in (0,\infty)$. The modified secant condition \eqref{MSCeq2} is essential to prove the global convergence of the MBFGS algorithm when dealing with non-convex objectives \cite{Zhou2006}. Note that the following expression holds irrespective of the line search used.
		\begin{align*}
			\langle z_n,s_n \rangle &= \langle y_n, s_n \rangle + \nu \|\nabla f(x_n)\|^r \|s_n\|^2 + \max \left\{-\frac{\langle s_n, y_n \rangle}{\|s_n\|^2}, 0\right\} \|s_n\|^2 \\
			& \geq \nu \|\nabla f(x_n)\|^r \|s_n\|^2 > 0 \text{ for all } n \in \mathbb{N}_0.
		\end{align*}
		
		
		Since \eqref{MSCeq2} leads to favorable theoretical properties, Faramarzi and Amini \cite{Faramarzi2019} modified the spectral CG method of Jian et al. \cite{Jian2017OMS} and proposed the following choice for the conjugate parameter:
		\begin{equation*}
			\beta_n^{ZDK} = \frac{\langle \nabla f(x_{n+1}), z_n \rangle}{\langle d_n, z_n \rangle} - \frac{\|z_n\|^2}{\la d_n,z_n\ra}\frac{\langle \nabla f(x_{n+1}), d_n \rangle}{\langle d_n, z_n \rangle} \text{ for all } n\in \mathbb{N}_0.
		\end{equation*}
		Based on this conjugate parameter, they constructed the following spectral parameter:
		\begin{equation*}
			\theta_{n+1}=\begin{cases}
				1- \frac{\|z_n\|^2}{\la d_n,z_n\ra}\frac{\la \nabla f(x_{n+1}),d_n\ra}{\la  \nabla f(x_{n+1}),z_n\ra}, & \text{ if } 	1- \frac{\|z_n\|^2}{\la d_n,z_n\ra}\frac{\la \nabla f(x_{n+1}),d_n\ra}{\la  \nabla f(x_{n+1}),z_n\ra} \in \left[\frac{1}{4p}+|q|+\eta, \tau\right],\\
				1, & \text{ otherwise},
			\end{cases}
		\end{equation*}
		where $\eta$ and $\tau$ are positive constants. To ensure global convergence of this method, Faramarzi and Amini \cite{Faramarzi2019} used the line search that satisfied the strong Wolfe conditions
		\begin{equation}\label{AFareq28}
			f(x_n+\alpha_n d_n) \leq f(x_n)+\delta \alpha_n \la \nabla f(x_n),d_n\ra,
		\end{equation}  
		\begin{equation}\label{AFareq29}
			|\la \nabla f(x_n+\alpha_n d_n), d_n\ra|\leq -\sigma \la \nabla f(x_n), d_n\ra,
		\end{equation}
		where $0<\delta <\sigma<1$, and the following assumptions:
		\begin{enumerate}\rm 
			\item[(A1)] The level set \(C_0=\{x\in \mathbb{R}^n: f(x)\leq f(x_0)\) is bounded.
			\item[(A2)] $f$ is continuously differentiable in some neighbourhood $\Omega$ of $C_0$ with $L$-Lipschitz continuous gradient.
		\end{enumerate}

		The primary motivation of the present study is to propose a modified spectral conjugate gradient method, guided by an eigenvalue-based analysis and a modified secant condition, and to prove global convergence for general objective functions under typical conditions.This new method always meets the sufficient-descent requirement, independent of the line search strategy applied.  To demonstrate the effectiveness of our proposed approach, we present numerical results that validate its performance and show its superior convergence compared to the modified spectral conjugate gradient method of Faramarzi and Amini \cite{Faramarzi2019}, and the scaled conjugate gradient method of Mard and Fakhari \cite{Mrad2024}. Furthermore, we explore its application to compressed sensing, emphasizing its practical relevance and significant advantages.
		

		
		\section{Preliminaries}
		
		In this section we present some fundamental results regarding Euclidean space that will be used to establish the convergence of our proposed method.
		
		\begin{lemma}\rm (\cite{Faramarzi2019}) \label{Lem1} Let $a,b,\sigma\in \mathbb{R}$. Then
			$$(c+\sigma d)^2\leq (1+\sigma^2)(c^2+d^2) \text{ for all } c,d,\sigma \geq 0.$$
		\end{lemma}
		
		
		\begin{lemma}\rm (\cite{Faramarzi2019}) \label{Lem2}
			Let $\mathbb{R}^m$ be a Euclidean space.  Then 
			$$\la u,v \ra \leq \frac{1}{2}(\|u\|^2+\|v\|^2) \text{ for all } u,v \in \mathbb{R}^m.$$
		\end{lemma}
		
		\begin{lemma}\rm \label{LemmaPre} 
			Let $\mathbb{R}^m$ be a Euclidean space, and let $\{u_n\}$ and $\{v_n\}$ be sequences in $\mathbb{R}^m$. Let $\sigma\in (0,\infty)$ and let $\{a_n\}$ and $\{b_n\}$ be two sequences in $(0,\infty)$ satisfying the following conditions:
			\begin{enumerate}
				\item[(C1)] $a_n>a>0$ for all $n \in \mathbb{N}$,
				\item[(C2)] $\la v_{n},u_{n}\ra <0$ for all $n \in \mathbb{N}$,
				\item[(C3)] $|\la v_{n+1},u_{n}\ra| \leq -\sigma \la v_{n},u_{n}\ra$ for all $n \in \mathbb{N}$.
			\end{enumerate}
			Suppose that
			\begin{equation}\label{pp1}
				u_{n+1}=-a_{n+1}v_{n+1}+b_nu_n \text{ for all } n \in \mathbb{N}.
			\end{equation}
			Then 
			\[\frac{\la v_{n+1}, u_{n+1}\ra^2}{\|u_{n+1}\|^2}+\frac{\la v_{n},u_n\ra^2}{\|u_n\|^2}\geq \frac{\|v_{n+1}\|^4}{\|u_{n+1}\|^2}\left[\frac{a^2}{1+\sigma^2}-a_{n+1}^2\frac{\la v_{n},u_n\ra^2}{\|u_n\|^2\|v_{n+1}\|^2}\right].\]
		\end{lemma}
		
		\begin{proof}
			Using \eqref{pp1} and condition (C2), we obtain 
			\begin{align*}
				b_n^2\|u_n\|^2&=\|u_{n+1}\|^2+ 2a_{n+1} \la  v_{n+1}, u_{n+1}\ra+a_{n+1}^2 \| v_{n+1}\|^2\\
				& \leq \|u_{n+1}\|^2+a_{n+1}^2\| v_{n+1}\|^2,
			\end{align*}
			which implies that 
			\begin{equation}\label{pp2}
				\frac{\|u_{n+1}\|^2}{\|u_n\|^2} \geq b_n^2-a_{n+1}^2 \frac{\| v_{n+1}\|^2}{\|u_n\|^2}.
			\end{equation}
			Again, using \eqref{pp1}, we have 
			\begin{equation}\label{ss1}
				-\la  v_{n+1}, u_{n+1}\ra +b_n \la  v_{n+1},u_n\ra =a_{n+1} \| v_{n+1}\|^2.
			\end{equation}
			It follows from condition (C3) that
			\begin{equation}\label{ss2}
				b_n \la v_{n+1},u_n\ra \leq b_n|\la v_{n+1},u_n\ra|  \leq -\sigma b_n \la v_{n},u_{n}\ra \leq \sigma b_n |\la v_{n},u_{n}\ra|.
			\end{equation}
			%
				Using \eqref{ss1} and \eqref{ss2}, we get
				\[	-\la v_{n+1},u_{n+1}\ra+\sigma b_n |\la u_n,v_{n}\ra| \geq a_{n+1} \|v_{n+1}\|^2.\]
			Squaring both sides and applying Lemma \ref{Lem1} 
			with $c=-\la v_{n+1},u_{n+1}\ra>0$ (from (C2)) and $d=b_n |\la v_n,u_n\ra|>0$ gives
			\begin{equation*}
				\la v_{n+1}, u_{n+1}\ra^2+b_n^2\la v_{n},u_n\ra^2\geq \frac{a_{n+1}^2}{1+\sigma^2}\|v_{n+1}\|^4.
			\end{equation*}
			Since  $a_{n+1}>a$, it follows that
			\begin{align*}
				\la v_{n+1}, u_{n+1}\ra^2+b_n^2\la v_{n},u_n\ra^2 & \geq \frac{a^2}{1+\sigma^2}\|v_{n+1}\|^4,
			\end{align*}
			which implies that 
			\begin{align}
				&	\frac{\la v_{n+1},u_{n+1}\ra^2}{\|u_{n+1}\|^2}+\frac{\la v_{n},u_n\ra^2}{\|u_n\|^2}\nonumber\\
				&=\frac{1}{\|u_{n+1}\|^2} \left[\la v_{n+1},u_{n+1}\ra^2+\frac{\|u_{n+1}\|^2}{\|u_n\|^2}\la v_{n},u_n\ra^2\right]\nonumber\\
				&\geq \frac{1}{\|u_{n+1}\|^2} \left[\frac{a^2}{1+\sigma^2}\|v_{n+1}\|^4+\left(\frac{\|u_{n+1}\|^2}{\|u_n\|^2}-b_n^2\right)\la v_n,u_n\ra^2\right].\label{pp4}
			\end{align}
			Using \eqref{pp2} and \eqref{pp4}, we conclude that
			\begin{equation*}
				\frac{\la v_{n+1}, u_{n+1}\ra^2}{\|u_{n+1}\|^2}+\frac{\la v_{n},u_n\ra^2}{\|u_n\|^2}\geq \frac{\|v_{n+1}\|^4}{\|u_{n+1}\|^2}\left[\frac{a^2}{1+\sigma^2}-a_{n+1}^2\frac{\la v_{n},u_n\ra^2}{\|u_n\|^2\|v_{n+1}\|^2}\right].
			\end{equation*}
%
			
		\end{proof}
		
		\begin{lemma}\rm (\cite{Zoutendijk}) \label{FarLem1}
			Let $f:\mathbb{R}^n \to \mathbb{R}$ be such that assumptions (A1) and (A2) hold. For any iterative method of the form \eqref{SearchDM} with $d_n$ being a descent direction and $\alpha_n$ determined by the strong Wolfe conditions \eqref{AFareq28} and \eqref{AFareq29}, we have
				\[	\sum_{n=0}^\infty \frac{\la \nabla f(x_n), d_n\ra^2}{\|d_n\|^2}<\infty.\]
			\end{lemma}\rm 
			\ \\
			It follows from Lemma \ref{FarLem1} that 
			\begin{equation}\label{Fareq36}
				\lim_{n \to \infty}\frac{\la \nabla f(x_n), d_n\ra^2}{\|d_n\|^2}=0.
			\end{equation}
			
			\section{A modified descent Dai-Liao spectral conjugate gradient method and its convergence analysis}
			
			
			In this section we introduce a modified spectral conjugate gradient (CG) descent method, inspired by ideas from quasi-Newton methods.

			Using the theoretical properties provided by \eqref{MSCeq2} and the Dai-Liao conjugate gradient parameter \eqref{DLmethod}, we propose the following modified Dai-Liao conjugate parameter:
			\begin{equation}\label{MDLmethod}
				\beta_n = \frac{\langle \nabla f(x_{n+1}), z_n \rangle}{\langle d_n, z_n \rangle} - t \frac{\langle \nabla f(x_{n+1}), s_n \rangle}{\langle d_n, z_n \rangle} \text{ for all } n \in \mathbb{N}_0,
			\end{equation}
			where $t$ is a nonnegative parameter.
			
			Note that, by \eqref{Search_direction} and \eqref{MDLmethod}, the search direction can be expressed as
			\begin{equation*}
				d_{n+1} = -Q_{n+1} \nabla f(x_{n+1}) \text{ for all } n \in \mathbb{N}_0,
			\end{equation*}
			where
			\[
			Q_{n+1} = I - \frac{s_n z_n^T}{\langle s_n, z_n \rangle} + t \frac{s_n s_n^T}{\langle s_n, y_n \rangle}.
			\]
			Following the methodology in \cite{KafakiGh2014}, we propose the following formula for the parameter $t$ in the modified Dai-Liao conjugate parameter \eqref{MDLmethod}:
			\begin{equation}\label{t_MDDL}
				t_n^{p,q} = p \frac{\|z_n\|^2}{\langle s_n, z_n \rangle} - q \frac{\langle s_n, z_n \rangle}{\|s_n\|^2} \text{ for all } n \in \mathbb{N}_0
			\end{equation}
			with $p \in \left( \frac{1}{4},\infty \right)$ and $q \in \left(-\infty, \frac{1}{4}\right)$. Using this parameter $t_n^{p,q}$, we define the following modified descent Dai-Liao (MDDL) conjugate parameter:
			\begin{equation}\label{CG_MDDL}
				\beta_n^{MDDL} = \frac{\langle \nabla f(x_{n+1}), z_n \rangle}{\langle d_n, z_n \rangle} - t_n^{p,q} \frac{\langle \nabla f(x_{n+1}), s_n \rangle}{\langle d_n, z_n \rangle} \text{ for all } n \in \mathbb{N}_0.
			\end{equation}
			Now, we consider the following spectral conjugate gradient search direction:
			\begin{equation}\label{Spect_direc}
				d_{n}=	\begin{cases}
					-\nabla f(x_0), & \text{ if } n=0, \\
					-\theta_{n} \nabla f(x_{n}) + \beta_{n-1}^{MDDL} d_{n-1} & \text{ if } n\geq 1,
				\end{cases}
			\end{equation}
			where $\{\theta_{n}\}$ is a sequence in $(0,\infty)$.

			\begin{proposition}\rm 
				Let $f:\mathbb{R}^m \to \mathbb{R}$ be a differentiable function and let $\{d_n\}$ be a sequence of spectral CG search directions generated by \eqref{Spect_direc}. Then the following statements hold:
				
				\begin{itemize}
					\item [(a)] For all  $ n \in \mathbb{N}_0,$ we have\begin{equation}\label{AFareq19}
						\la \nabla f(x_{n}),d_{n} \ra \leq -\kappa_n\|\nabla f(x_{n})\|^2,
					\end{equation}
					where \begin{equation*}
						\kappa _{n}=\left\{ 
						\begin{array}{ll}
							1, & \text{if }n=0, \\ 
							\theta _{n}-\left( \frac{1}{4p}+|q|\right),  & \text{if }n\geq 1.%
						\end{array}%
						\right. 
					\end{equation*}%
					\item [(b)]   If $\theta_{n}>\frac{1}{4p}+|q|$ for all $n \in \mathbb{N}$, then the direction given by \eqref{Spect_direc} is a descent direction.
				\end{itemize}
				
			\end{proposition}

			\begin{proof}(a)
				Note that $\la \nabla f(x_0),d_0\ra =-\|\nabla f(x_0)\|^2$.  Therefore inequality \eqref{AFareq19} holds for $n=0$. We claim that the inequality \eqref{AFareq19} holds for $n>0$. Let $n\in \mathbb{N}_0$. 	Note that  $p \in \left( \frac{1}{4},\infty \right)$ and $q \in \left(-\infty, \frac{1}{4}\right)$. For simplicity, let us denote $\nabla f(x_{n+1})=F_{n+1}$. Then, using \eqref{t_MDDL}, \eqref{CG_MDDL} and \eqref{Spect_direc}, we get
				\begin{align*}
					&\la F_{n+1},d_{n+1}\ra \\
					&=\la F_{n+1},-\theta_{n+1} F_{n+1}+\beta_{n}^{MDDL}d_{n}\ra\\
					&=-\theta_{n+1}\|F_{n+1}\|^2+\frac{\la F_{n+1},z_{n}\ra \la F_{n+1},d_{n}\ra}{\la d_{n},z_{n}\ra}-\frac{t_n^{p,q}\la s_{n},F_{n+1}\ra \la d_{n},F_{n+1}\ra}{\la d_{n},z_{n}\ra}\\
					&= -\theta_{n+1}\|F_{n+1}\|^2+\frac{\la F_{n+1},z_n\ra \la F_{n+1},d_{n}\ra}{\la d_{n},z_n\ra}\\
					&\quad-\left(p\frac{\|z_n\|^2}{\la s_n,z_n\ra}-q\frac{\la s_n,z_n\ra}{\|s_n\|^2}\right)\frac{\la s_{n},F_{n+1}\ra \la d_{n},F_{n+1}\ra}{\la d_{n},z_n\ra}\\
					&= -\theta_{n+1}\|F_{n+1}\|^2+\frac{\la F_{n+1},z_n\ra \la F_{n+1},d_{n}\ra}{\la d_{n},z_n\ra}-p\frac{\|z_n\|^2}{\la s_n,z_n\ra}\frac{\la s_{n},F_{n+1}\ra \la d_{n},F_{n+1}\ra}{\la d_{n},z_n\ra} \\
					& \quad+q\frac{\la s_n,z_n\ra}{\|s_n\|^2}\frac{\la s_{n},F_{n+1}\ra \la d_{n},F_{n+1}\ra}{\la d_{n},z_n\ra}.
				\end{align*}

				Since $s_n=\alpha_n d_n$, we have
				\allowdisplaybreaks
				\begin{align}
					\la 	F_{n+1},d_{n+1} \ra
					& = -\theta_{n+1}\|F_{n+1}\|^2+\frac{\la F_{n+1},z_n\ra \la F_{n+1},d_{n}\ra}{\la d_{n},z_n\ra}\nonumber\\
					& \quad-p\frac{\|z_n\|^2\la d_{n},F_{n+1}\ra^2}{\la d_{n},z_n\ra} +q\frac{\la d_n,z_n\ra}{\|d_n\|^2}\frac{\la d_{n},F_{n+1}\ra^2}{\la d_{n},z_n\ra} \nonumber\\
					&= -\theta_{n+1}\|F_{n+1}\|^2+\frac{\la F_{n+1},z_n\ra \la F_{n+1},d_{n}\ra}{\la d_{n},z_n\ra}\nonumber\\
					& \quad -p\frac{\|z_n\|^2\la d_{n},F_{n+1}\ra^2}{\la d_{n},z_n\ra^2}+q\frac{\la d_{n},F_{n+1}\ra^2}{\|d_{n}\|^2}. \label{AFareq21''}
				\end{align}
				Using Lemma \ref{Lem2} with
				$u_n=\frac{1}{\sqrt{2p}}F_{n+1}$ and $v_n=\frac{\sqrt{2p}\la F_{n+1},d_{n}\ra z_n}{\la d_{n},z_n\ra},$ 
				we obtain
				\begin{equation}\label{AFareq21}
					\frac{\la F_{n+1},d_{n}\ra \la F_{n+1},z_n\ra}{\la d_{n},z_n\ra} \leq \frac{1}{4p}\|F_{n+1}\|^2+\frac{p\|z_n\|^2\la d_{n},F_{n+1}\ra^2}{\la d_{n},z_n\ra^2}.
				\end{equation}
				On the other hand, by the Cauchy-Schwarz inequality, we have
				\begin{equation}\label{AFareq21'}
					q\frac{\la d_{n},F_{n+1}\ra^2}{\|d_{n}\|^2}\leq|q| \frac{\|d_{n}\|^2\|F_{n+1}\|^2}{\|d_{n}\|^2} = |q|\|F_{n+1}\|^2.
				\end{equation}
				Combining \eqref{AFareq21''}, \eqref{AFareq21} and \eqref{AFareq21'}, we conclude that
				\begin{align*}
					\la F_{n+1},d_{n+1}\ra&\le-\theta_{n+1}\|F_{n+1}\|^2+\frac{1}{4p}\|F_{n+1}\|^2+ |q|\|F_{n+1}\|^2\\
					&=-\left(\theta_{n+1}-\left(\frac{1}{4p}+|q|\right)\right)\|F_{n+1}\|^2.
				\end{align*}
				(b) Suppose that $\theta_{n}>\frac{1}{4p}+|q|$ for all $n \in \mathbb{N}$.     Then, employing \eqref{AFareq19}, we get $\la \nabla f(x_{n}),d_{n}\ra < 0$ for all $n\in \mathbb{N}_0$. Therefore, $\{d_n\}$ is a descent direction as claimed.
			\end{proof}

			Now, we construct the spectral parameter for our conjugate gradient method based on the conjugate parameter $\beta_n^{MDDL}$.
			
			Quasi-Newton methods are a prominent class of iterative optimization methods that avoid computing the exact Hessian at every step. Their main goal is to retain the fast convergence associated with Newton’s method while keeping the computational burden low. Rather than evaluating the true Hessian, these methods update an approximate one at each iteration. Motivated by the structure of quasi-Newton directions, the search direction $d_n$ in \eqref{Spect_direc} is constructed to emulate this behavior. Our objective, therefore, is to choose an appropriate spectral parameter $\theta_{n+1}$ such that 			
			$$-B_{n+1}^{-1}\nabla f(x_{n+1}) \approx -\theta_{n+1}\nabla f(x_{n+1})+\beta_n^{MDDL}d_n,$$
			where $B_{n+1}$ is an approximation of the Hessian matrix $\nabla^2f(x_{n+1})$.
			Taking the inner product with $B_{n+1}s_n$, we find that 
			$$-\la s_n,\nabla f(x_{n+1}) \ra \approx -\theta_{n+1}\la B_{n+1}s_n,\nabla f(x_{n+1})\ra+\beta_n^{MDDL}\la B_{n+1}s_n,d_n\ra,$$
			which implies that 
			\begin{equation}\label{AFareq222}
			\theta_{n+1}\approx \frac{1}{\la B_{n+1}s_n,\nabla f(x_{n+1})\ra}(\la s_n,\nabla f(x_{n+1})\ra +\beta_n^{MDDL}\la B_{n+1}s_n,d_n\ra).
			\end{equation}
			From \eqref{MSCeq1}, $B_{n+1}$ must satisfy the modified secant condition $B_{n+1}s_n=z_n$. It follows from \eqref{AFareq222} that

			\begin{align}
				{\theta_{n+1}}
				&\approx \frac{1}{\la z_n,\nabla f(x_{n+1})\ra}(\la s_n,\nabla f(x_{n+1})\ra+\beta_n^{MDDL}\la z_n,d_n\ra)\nonumber\\
				& =\frac{1}{\la z_n,\nabla f(x_{n+1})\ra}\left(\la s_n,\nabla f(x_{n+1})\ra+\left( \frac{\la\nabla f(x_{n+1}), z_n\ra}{\la d_n, z_n\ra } \right. \right. \nonumber\\
				& \quad  \left. \left. -t_n^{p,q} \frac{\la \nabla f(x_{n+1}),s_n\ra}{\la d_n,z_n\ra} \right) \la z_n,d_n\ra\right)\nonumber\\
				& =\frac{1}{\la z_n,\nabla f(x_{n+1})\ra}\left(\la s_n,\nabla f(x_{n+1})\ra+	\frac{\la \nabla f(x_{n+1}), z_n\ra \la z_n,d_n\ra}{\la d_n, z_n\ra} \right.\nonumber\\
				& \quad \left.  -t_n^{p,q} \frac{\la \nabla f(x_{n+1}),s_n\ra\la z_n,d_n\ra }{\la d_n,z_n\ra} \right)\nonumber\\
				& =\frac{\la s_n,\nabla f(x_{n+1})\ra+	\la \nabla f(x_{n+1}), z_n\ra-t_n^{p,q} \la \nabla f(x_{n+1}),s_n\ra }{\la z_n,\nabla f(x_{n+1})\ra}\nonumber\\
				&=1-\frac{{t_n^{p,q} \la s_n, \nabla f(x_{n+1})\ra }- \la s_n,\nabla f(x_{n+1})\ra}{\la z_n, \nabla f(x_{n+1})\ra} \label{AFareq23'}\\
				&= 1-(t_n^{p,q}-1)\frac{\la s_n,\nabla f(x_{n+1})\ra}{\la z_n, \nabla f(x_{n+1})\ra}. \label{N+}
			\end{align}
		 Neglecting $\la s_n,\nabla f(x_{n+1})\ra$ in
			\eqref{AFareq23'}, we arrive at
			\begin{equation}\label{N-}
				{\theta_{n+1}}=1-\frac{t_n^{p,q} \la s_n, \nabla f(x_{n+1})\ra}{\la z_n, \nabla f(x_{n+1})\ra}.
			\end{equation}
			We denote the  $\theta_{n+1}$ obtained in equation \eqref{N+} by $\theta_{n+1}^{R}$ and the $\theta_{n+1}$ obtained in equation \eqref{N-} by $\theta_{n+1}^{N}$.
To ensure that the sufficient-descent condition holds and to keep the spectral parameters bounded, we introduce the following choices for the spectral parameters:
			\begin{equation}\label{AFareq251}
				\theta_{n+1}=\begin{cases}
					\theta_{n+1}^{R}, & \text{ if } \theta_{n+1}^{R}\in \left[\frac {1}{4p}+|q|+\eta, \tau\right],\\
					1, & \text{ otherwise }
				\end{cases}
			\end{equation}
			or
			\begin{equation}\label{AFareq261}
				\theta_{n+1}=\begin{cases}
					\theta_{n+1}^{N}, & \text{ if } \theta_{n+1}^{N}\in \left[\frac{1}{4p}+|q|+\eta, \tau\right],\\
					1, & \text{ otherwise },
				\end{cases}
			\end{equation}
			where $\eta$ and $\tau$ are positive constants. Relations \eqref{AFareq251} and \eqref{AFareq261} lead to  
			\begin{equation}\label{AFareq261a}
				\frac{1}{4p}+|q|+\eta \leq \theta_{n+1} \leq \tau\textrm{ for all }n\in \mathbb{N}_0.
			\end{equation}

			Now, we propose the following modified descent Dai-Liao spectral conjugate gradient method (MDDLSCG):

			%
			
			\begin{algorithm}[H]
				\label{MDDLSCG}
				\caption{MDDLSCG Algorithm}
				\textbf{Input:} Let $x_0 \in \mathbb{R}^n$ and let $\epsilon, \delta,  \sigma, \eta, \tau, r,\nu \in (0,\infty)$ be such that  \\
				\qquad \quad  ~ $0 < \delta < \sigma < 1$.\\
				\textbf{Step 0:} Set $d_0 = -\nabla f(x_0)$.\\
				\textbf{Step 1:} If $\|\nabla f(x_n)\|_\infty < \varepsilon$, stop, where
				$\|u\|_\infty = \max_{1 \le i \le m} |u_i|$ for\\
				\qquad \qquad $u = (u_1, \ldots, u_m) \in \mathbb{R}^m$.\\
				\textbf{Step 2:}	Determine the step length $\alpha_n$ satisfying the strong Wolfe \\
				\qquad \qquad conditions \eqref{AFareq28} and \eqref{AFareq29}.\\
				\textbf{Step 3:} Set $x_{n+1} = x_n + \alpha_n d_n$.\\
				\textbf{Step 4:} Compute $d_{n+1}$ by
				\[d_{n+1}=-\theta_{n+1}\nabla f(x_{n+1})+\beta_n^{MDDL}d_{n} \text{ for all } n \in \mathbb{N}_0,\]
				\qquad \qquad  where the CG parameter $\beta_n^{MDDL}$ is given by \eqref{CG_MDDL} and \\
					\qquad \qquad ~  the spectral 
			 parameter $\theta_n$ is given by \eqref{AFareq251} or \eqref{AFareq261}.\\
				\textbf{Step 5:} Set $n = n + 1$ and return to Step 1.
			\end{algorithm}
			
			
			\begin{remark}\rm\label{Remark2}(i)
				Using $\theta_{n}\geq \frac{1}{4p}+|q| + \eta$ and \eqref{AFareq19}, we conclude that
				\begin{equation}\label{AFareq27}
					\la \nabla f(x_{n}),d_{n}\ra \leq -\eta\|\nabla f(x_{n})\|^2\text{ for all } n \in \mathbb{N}_0.
				\end{equation}
				(ii) If $\|\nabla f(x_n)\|\ne0$, then using
				\eqref{AFareq29} and \eqref{AFareq27}, we see that
				\begin{equation}\label{Fareq30}
					\la d_n,y_n\ra\geq -(1-\sigma) \la \nabla f(x_n),d_n\ra\geq (1-\sigma)\eta\|\nabla f(x_n)\|^2>0,
				\end{equation}
				which, by (\ref{MSCeq2}), implies  that $h_n=\nu$ and hence
				\begin{equation}\label{Fareq31}
					z_n=y_n+\nu\|\nabla f(x_n)\|^rs_n.
				\end{equation}
			\end{remark}

			Now, we analyze the global convergence  of the 
			MDDLSCG Algorithm. 
			
			\begin{remark}\rm 
				\begin{enumerate}
					\item[(i)] Employing \eqref{AFareq28} and \eqref{AFareq27}, we can see that 
					$\{f(x_n)\}$ is decreasing. Thus, the sequence $\{x_n\}$ generated by Algorithm \ref{MDDLSCG} is contained in $C_0$.
					
					\item[(ii)] From Assumption (A1), there exists $M>0$ such that $\|x\|\leq M$ for all $x\in C_0$. Then using Assumption (A2), we have 
					\begin{align}\label{Fareq34}
						\Vert \nabla f(x)\Vert & \leq \Vert \nabla f(x)-\nabla f(x_{0})\Vert +\Vert
						\nabla f(x_{0})\Vert \notag \\
						& \leq L\Vert x-x_{0}\Vert +\Vert \nabla f(x_{0})\Vert  \notag  \\
						& \leq 2ML+\Vert \nabla f(x_{0})\Vert =\gamma \text{ for all }x\in C_{0}.
					\end{align}%
				\end{enumerate}
			\end{remark}
			
			In addition, suppose that $\nabla f(x_n)\neq 0$ for all $n\in \mathbb{N}_0$; otherwise, a stationary point will be found.

			\begin{lemma}\rm \label{FarLem2}
				Let $f:\mathbb{R}^n \to \mathbb{R}$ be a function and let $\{x_n\}$ be the   sequence generated by Algorithm \ref{MDDLSCG} such that Assumptions (A1) and (A2) hold.  Suppose that $\theta_n >\frac{1}{4p}+|q|$ for all $n\in \mathbb{N}$.  If $	\liminf_{n \to \infty} \|\nabla f(x_n)\| \neq 0$, then  $\sum_{n=0}^\infty \frac{\|\nabla f(x_n)\|^4}{\|d_n\|^2}<\infty.$ 
			\end{lemma}
			\begin{proof}
				Since $	\liminf_{n \to \infty} \|\nabla f(x_n)\| \neq 0$, there exists $\epsilon>0$ such that $\|\nabla f(x_n)\|\geq \epsilon$ for all $n \in \mathbb{N}_0$. When combined with \eqref{Fareq36}, this implies  that 
				\begin{equation}\label{Fareq39}
					\lim_{n \to \infty} \frac{\la \nabla f(x_n), d_n\ra^2}{\|d_n\|^2 \|\nabla f(x_{n+1})\|^2} =0.
				\end{equation}
				Let  $u_n=d_n$, $v_n=\nabla f(x_n)$, $a_n=\theta_n$ and $b_n=\beta_n$ in Lemma \ref{LemmaPre}.  Since $\theta_{n}>\frac{1}{4p}+|q|>0$, it follows from Remark \ref{Remark2} that   $\la \nabla f(x_n),d_n\ra<0$ for all $n\in \mathbb{N}_0$. Also from strong Wolfe condition \eqref{AFareq29}, we have $|\la \nabla f(x_{n+1}), d_n\ra|\leq -\sigma \la \nabla f(x_n), d_n\ra$ for all $n\in \mathbb{N}_0$.
				Then, clearly all the conditions (C1), (C2) and (C3) of Lemma \ref{LemmaPre} are satisfied. Thus, we have
				\begin{equation*}
					\frac{\la \nabla f(x_{n+1}), d_{n+1}\ra^2}{\|d_{n+1}\|^2}+\frac{\la \nabla f(x_n),d_n\ra^2}{\|d_n\|^2}  \geq \frac{\|\nabla f(x_{n+1})\|^4}{\|d_{n+1}\|^2}\left[B- \theta_{n+1}^2\frac{\la \nabla f(x_n),d_n\ra^2}{\|d_n\|^2\|\nabla f(x_{n+1})\|^2}\right],
				\end{equation*}
				where $B=\left(\frac{1}{4p}+|q|\right)^2\frac{1}{1+\sigma^2}$.
				It follows from \eqref{Fareq39} that for all sufficiently large $n$, we have
				\begin{equation}\label{Fareq44}
					\frac{\la \nabla f(x_{n+1}), d_{n+1}\ra^2}{\|d_{n+1}\|^2}+\frac{\la \nabla f(x_n),d_n\ra^2}{\|d_n\|^2}  \geq B\frac{\|\nabla f(x_{n+1})\|^4}{\|d_{n+1}\|^2}.
				\end{equation}
				Using  \eqref{Fareq44} and  Lemma \ref{FarLem1}, we get  that
				$\sum_{n=1}^\infty\frac{\|\nabla f(x_{n+1})\|^4}{\|d_{n+1}\|^2}<\infty.$ 
			\end{proof}

			\begin{theorem}\rm 
				Let $f:\mathbb{R}^n \to \mathbb{R}$ be  a real function and let  $\{x_n\}$ be the  sequence generated by Algorithm \ref{MDDLSCG} such that Assumptions (A1) and (A2) hold. Then 
				$
				\liminf_{n \to \infty} \|\nabla f(x_n)\|=0.
				$
			\end{theorem}
			\begin{proof}
				Suppose, for the sake of contradiction, that the assertion of the theorem does not hold. Then there exists a positive constant $\xi$ such that
				\begin{equation}\label{Fareq47}
					\|\nabla f(x_n)\|>\xi \text{ for all } n \in \mathbb{N}_0.
				\end{equation}
				Employing \eqref{Fareq30}, \eqref{Fareq31} and \eqref{Fareq47}, we obtain
				\begin{align}
					\la d_n,z_n\ra&=\la d_n,y_n\ra+\nu\|\nabla f(x_n)\|^r\alpha_n\|d_n\|^2\nonumber\\
					&\geq \nu\|\nabla f(x_n)\|^r\alpha_n\|d_n\|^2\geq \nu\xi^r\alpha_n \|d_n\|^2. \label{Fareq48}
				\end{align}
				Using  Assumption (A2), \eqref{Fareq31} and \eqref{Fareq34}, we have
				\begin{equation}\label{Fareq49}
					\|z_n\|\leq \|y_n\|+\nu\|\nabla f(x_n)\|^r\|s_n\|\leq L\|s_n\|+\nu \gamma^r\|s_n\|=(L+\nu\gamma^r)\|s_n\|.
				\end{equation}
				Thus, relations \eqref{CG_MDDL}, \eqref{Fareq34}, \eqref{Fareq48}, \eqref{Fareq49}, along with the Cauchy-Schwarz inequality, yield
				\allowdisplaybreaks
				\begin{align}
					&	\|\beta_n^{MDDL}\|\nonumber\\
					&=\left\|\frac{\la \nabla f(x_{n+1}),z_n\ra}{\la d_n,z_n\ra}-t_n^{p,q}\frac{\la \nabla f(x_{n+1}),s_n\ra}{\la d_n,z_n\ra}\right\|\nonumber\\
					&=\left\|\frac{\la \nabla f(x_{n+1}),z_n\ra}{\la d_n,z_n\ra}-\left(p\frac{\|z_n\|^2}{\la s_n,z_n\ra}-q\frac{\la s_n,z_n\ra}{\|s_n\|^2}\right)\frac{\la \nabla f(x_{n+1}),s_n\ra}{\la d_n,z_n\ra}\right\|\nonumber\\
					&=\left\|\frac{\la \nabla f(x_{n+1}),z_n\ra}{\la d_n,z_n\ra}-p\frac{\|z_n\|^2}{\la s_n,z_n\ra}\frac{\la \nabla f(x_{n+1}),s_n\ra}{\la d_n,z_n\ra}+q\frac{\la s_n,z_n\ra}{\|s_n\|^2}\frac{\la \nabla f(x_{n+1}),s_n\ra}{\la d_n,z_n\ra}\right\|\nonumber\\
					&=\left\|\frac{\la \nabla f(x_{n+1}),z_n\ra}{\la d_n,z_n\ra}\right\|+p\left\|\frac{\|z_n\|^2}{\la s_n,z_n\ra}\frac{\la \nabla f(x_{n+1}),s_n\ra}{\la d_n,z_n\ra}\right\| \nonumber\\
					& \quad +|q|\left\|\frac{\la s_n,z_n\ra}{\|s_n\|^2}\frac{\la \nabla f(x_{n+1}),s_n\ra}{\la d_n,z_n\ra}\right\|\nonumber\\
					&=\frac{\|\nabla f(x_{n+1})\|\|z_n\|}{\la d_n,z_n\ra}+p\frac{\|z_n\|^2\|\nabla f(x_{n+1})\|\|d_n\|}{\la d_n,z_n\ra^2}+|q|\frac{\|\nabla f(x_{n+1})\|}{\|d_n\|} \label{Fareq50} \\ 
					&\leq \frac{\gamma(L+\nu\gamma^r)}{\nu\xi^r\|d_n\|}+\frac{p\gamma(L+\nu\gamma^r)^2}{\nu^2\xi^{2r}\|d_n\|}+|q|\frac{\gamma}{\|d_n\|}.\nonumber
				\end{align}
				Note that  $\theta_{n+1}\leq \tau$ by \eqref{AFareq261a}. Hence, using \eqref{Spect_direc} and \eqref{Fareq50}, we arrive at
				\begin{align}
					\|d_{n+1}\|&=\|-\theta_{n+1}\nabla f(x_{n+1})+\beta_n^{MDDL}d_n\| \nonumber\\
					& \leq \|\theta_{n+1}\|\|\nabla f(x_{n+1})\|+\|\beta_n^{MDDL}\|\|d_n\| \nonumber\\
					& \leq \tau \gamma + \frac{\gamma(L+\nu\gamma^r)}{\nu\xi^r}+\frac{p\gamma(L+\nu\gamma^r)^2}{\nu^2\xi^{2r}}+|q|\gamma. \label{Fareq51}
				\end{align}
				Using \eqref{Fareq51}, we observe that 
				$$\frac{1}{\|d_{n+1}\|^2}\geq \frac{1}{E^2}\textrm{ for all }n\in \mathbb{N}_0,$$
				where $E=\tau \gamma + \frac{\gamma(L+\nu\gamma^r)}{\nu\xi^r}+\frac{p\gamma(L+\nu\gamma^r)^2}{\nu^2\xi^{2r}}+|q|\gamma$. Thus, 
				\begin{equation}\label{Fareq45}
					\sum_{n=0}^{\infty} \frac{1}{\|d_n\|^2}=\infty.
				\end{equation}
				Employing \eqref{Fareq47} and  Lemma \ref{FarLem2}, we find that
				$$\sum_{n=0}^\infty \frac{1}{\|d_n\|^2}=\sum_{n=0}^\infty \frac{\|\nabla f(x_n)\|^4}{\|d_n\|^2 \|\nabla f(x_n)\|^4}\leq \frac{1}{\xi^4}\sum_{n=0}^\infty \frac{\|\nabla f(x_n)\|^4}{\|d_n\|^2}<\infty,$$
				which contradicts \eqref{Fareq45}. Therefore, $\liminf_{n \to \infty}\|\nabla f(x_n)\|=0$, as asserted.
			\end{proof}

			\section{Numerical Experiments}
			
			In this section we present the numerical performance of our proposed method (MDDLSCG), comparing it with Algorithm 3.1 in \cite{Faramarzi2019} (denoted by MSCG) and the scaled conjugate gradient method (denoted by ScCG) in \cite{Mrad2024}. All numerical experiments were conducted in MATLAB R2024b on an HP workstation with Intel Core i7-13700 processor, 32 GB RAM and Windows 11.
			

			For the numerical experiments, we consider two test problems: (i) finding the minima of the Beale function, and (ii)  finding the minima of the Almost Perturbed Quadratic (APQ) function.
			\\   \\
			\textbf{Beale Test Function:} The Beale function is multimodal with sharp peaks at the corners of the input domain, which is defined from the general Beale function as follows:
			
			\begin{footnotesize}
				\[
				f(x, y) = \left( 1.5 - x + xy \right)^2 + \left( 2.25 - x + xy^2 \right)^2 + \left( 2.625 - x + xy^3 \right)^2 \text{ for all } (x,y) \in \mathbb{R}^2.
				\]
			\end{footnotesize}
			The global minimum is located at \( (x, y) = (3, 0.5) \), where \( f(x, y) = 0 \).
			
			\begin{figure}[htpb]
				\centering
				\subfigure[Surface plot of the Beale function]{
					\includegraphics[width=60mm]{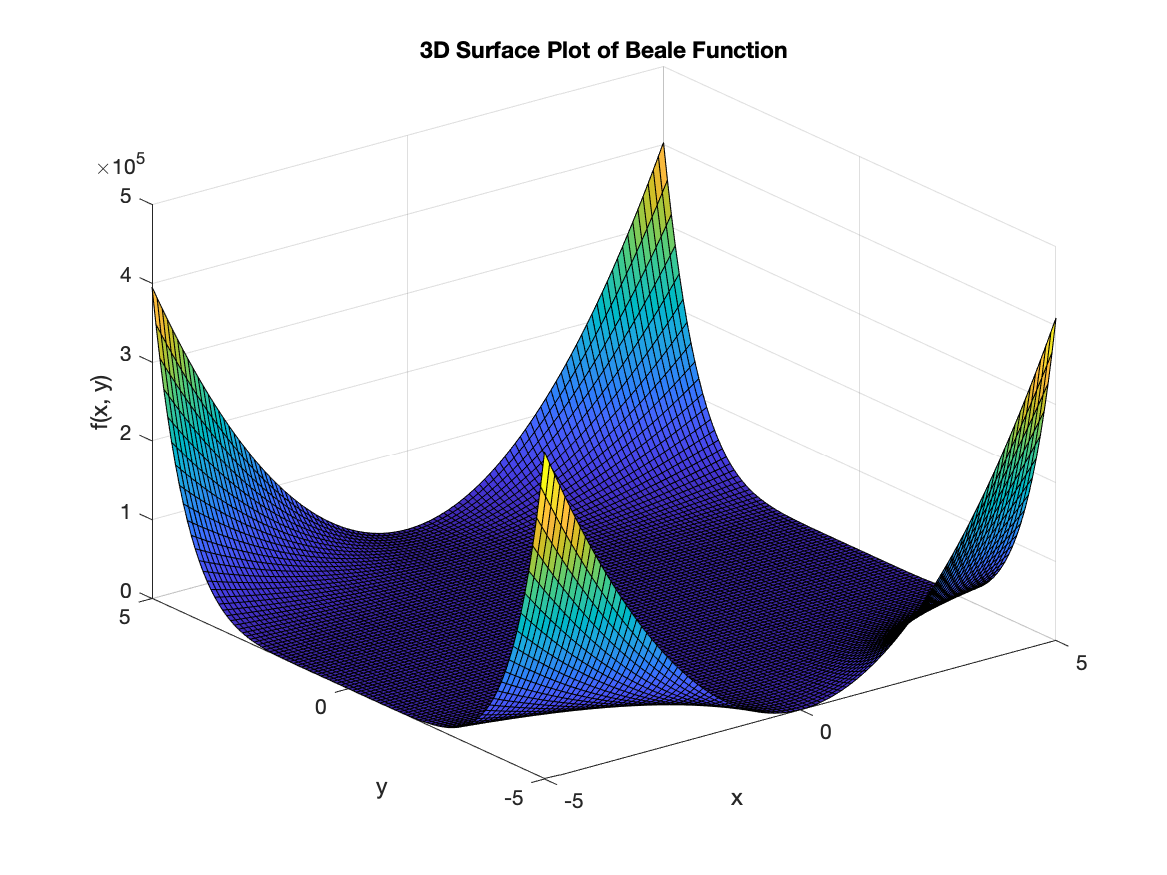}
				}%
				\subfigure[Contour plot of the Beale function]{
					\includegraphics[width=60mm]{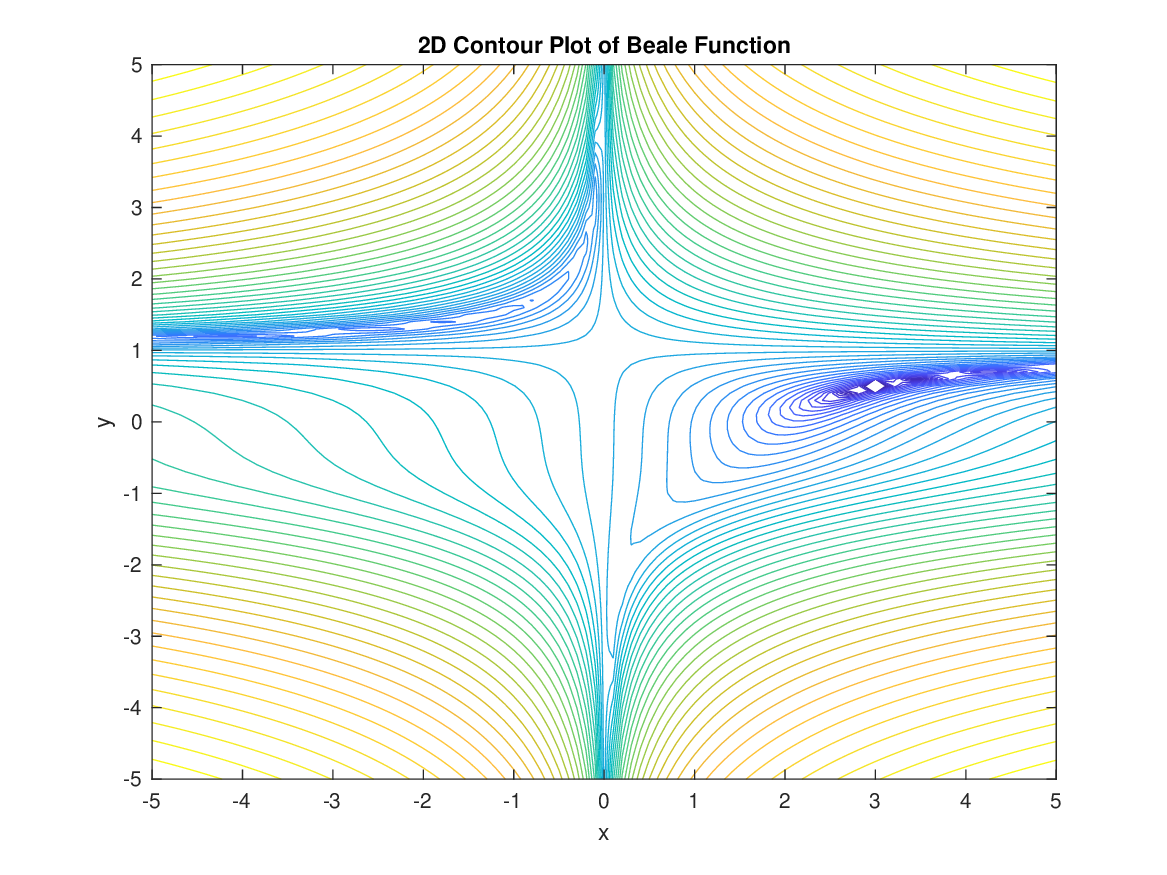}
				}
				\caption{Surface and Contour plot of  the Beale functions}
				\label{beale}
			\end{figure}

			For the numerical tests, we choose $\delta= 0.01$, $\sigma \in  \{0.1,0.2,0.4,0.6,0.8,0.9\}$ for all the algorithms, $\eta = 0.001$, $\tau = 10$, $r=1$ and $\nu=0.001$ for algorithms MSCG and MDDLSCG and $p=0.4$ and $q=0.2$ for the algorithm MDDLSCG. We take the initial point $x_0=(1,0.8)$.
			The iterations are terminated once they satisfy the stopping criterion $E_n=\|\nabla f(x_n)\|_\infty<10^{-15}$.
			
			The numerical results of all the algorithms MDDLSCG, MSCG and ScCG for the Beale function are summarized in Table \ref{Tab}, while their performance is illustrated in Figure \ref{bealefig1} and \ref{bealefig2}. In Figures \ref{bealefig1} (a), (c), (e), the iteration paths are visualized on the contour plot of the Beale function, showing how each method converges to the minimizer for $\sigma =0.1, 0.2, 0.6$, respectively. Figures \ref{bealefig1} (b), (d), (f) depict the convergence rate by plotting the norm of the gradient, $\|\nabla f(x_n)\|_\infty$ versus the number of iterations on a logarithmic scale for $\sigma =0.1, 0.2, 0.6$, respectively.  The comparison graph between different values of sigma and the number of iterations is depicted in  Figure \ref{bealefig2} (a) and the comparison graph between different values of sigma and execution time is in  Figure \ref{bealefig2} (b).
			

			In view of the data presented, it is evident that our proposed method, MDDLSCG, outperforms both MSCG and ScCG in all tested scenarios. As shown in Table \ref{Tab}, Figure \ref{bealefig1} and \ref{bealefig2} (b), MDDLSCG reaches the stopping criteria in fewer iterations and requires less execution time.  Additionally, using Figure \ref{bealefig2} (a), we observe that these methods are significantly sensitive to variations in the parameter \( \sigma \), which affects their performance.

			\begin{table}[h]
				\centering
				\begin{adjustbox}{width=1\textwidth}
					\begin{tabular}{|c|ccc|ccc|ccc|}
						\hline
						$\sigma$& & MDDLSCG&&&MSCG &&& ScCG &\\
						&  itr& Tcpu& $E_n$ &  itr& Tcpu& $E_n$ &  itr& Tcpu& $E_n$\\
						\hline 
						0.1  &37 &1.86e-02 &9.49e-15&63 &3.14e-02 &2.63e-16&200 & 5.40e-02 &5.27e-16\\
						0.2  &36 &1.75e-02 &3.16e-15 &63&2.76e-02 &9.07e-15&136 &3.99e-02 &4.51e-15\\
						
						0.4  &53 &2.44e-02 &9.40e-15&73 &2.65e-02 &8.75e-15 &126 &4.73e-02 &2.87e-15\\
						
						0.6 &40 &2.36e-02 &7.91e-16 &72 &2.78e-02 &5.74e-15 &95 &4.94e-02 &9.32e-15\\
						
						0.8 &67 &3.24e-02&5.32e-15 &111 &2.86e-02 &4.48e-15&201 &4.92e-02&6.91e-15\\
						
						0.9 &64 &3.28e-02 &8.16e-15 &101 &3.48e-02 &8.74e-15 &184 &4.37e-02&1.08e-15\\
						\hline
					\end{tabular}
				\end{adjustbox}
				\caption{Comparison table for MDDLSCG,  MSCG and ScCG for Beale function}
				\label{Tab}
			\end{table}

			\begin{figure}[htbp]
				\centering
				\subfigure[Behavior of iterations paths for $\sigma =0.1$]{
					\includegraphics[width=65mm]{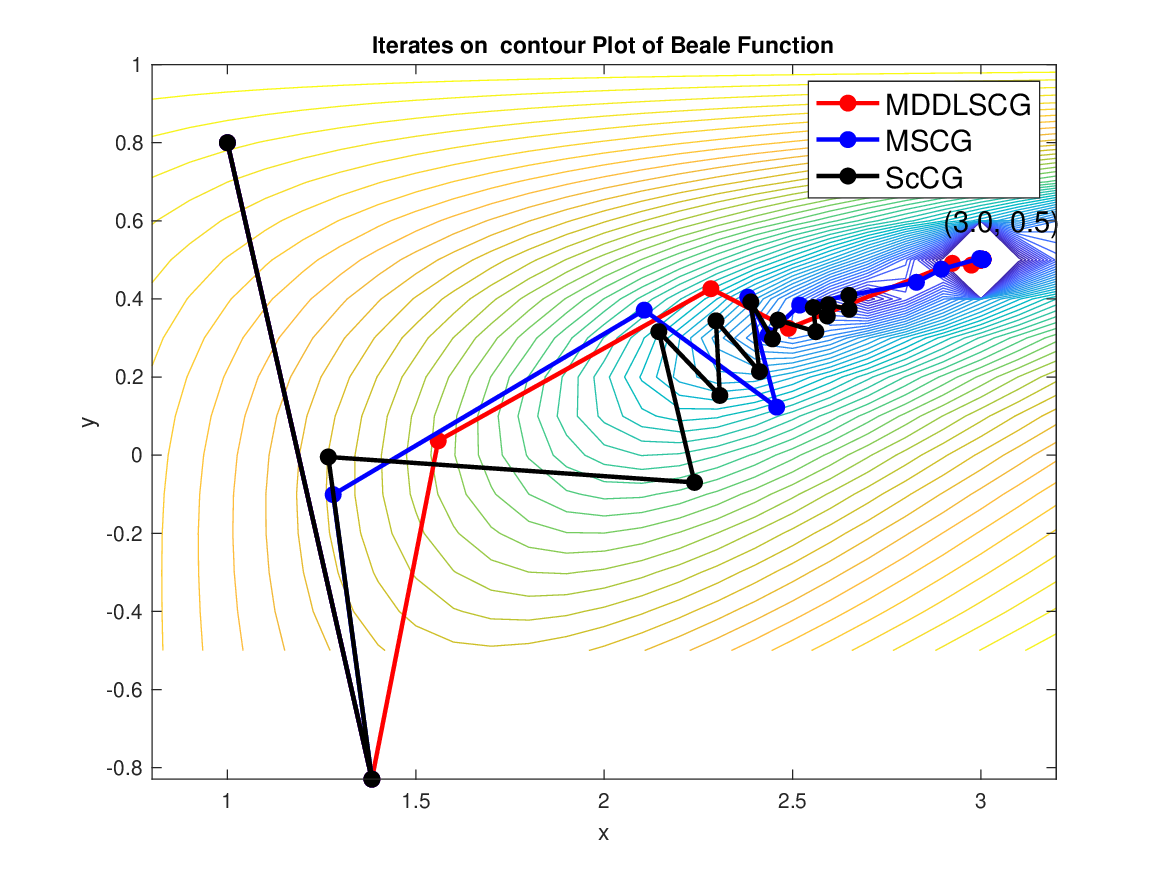}
					z		}%
				\subfigure[Behavior of $\|\nabla f(x_n)\|_\infty$ for $\sigma =0.1$]{
					\includegraphics[width=65mm]{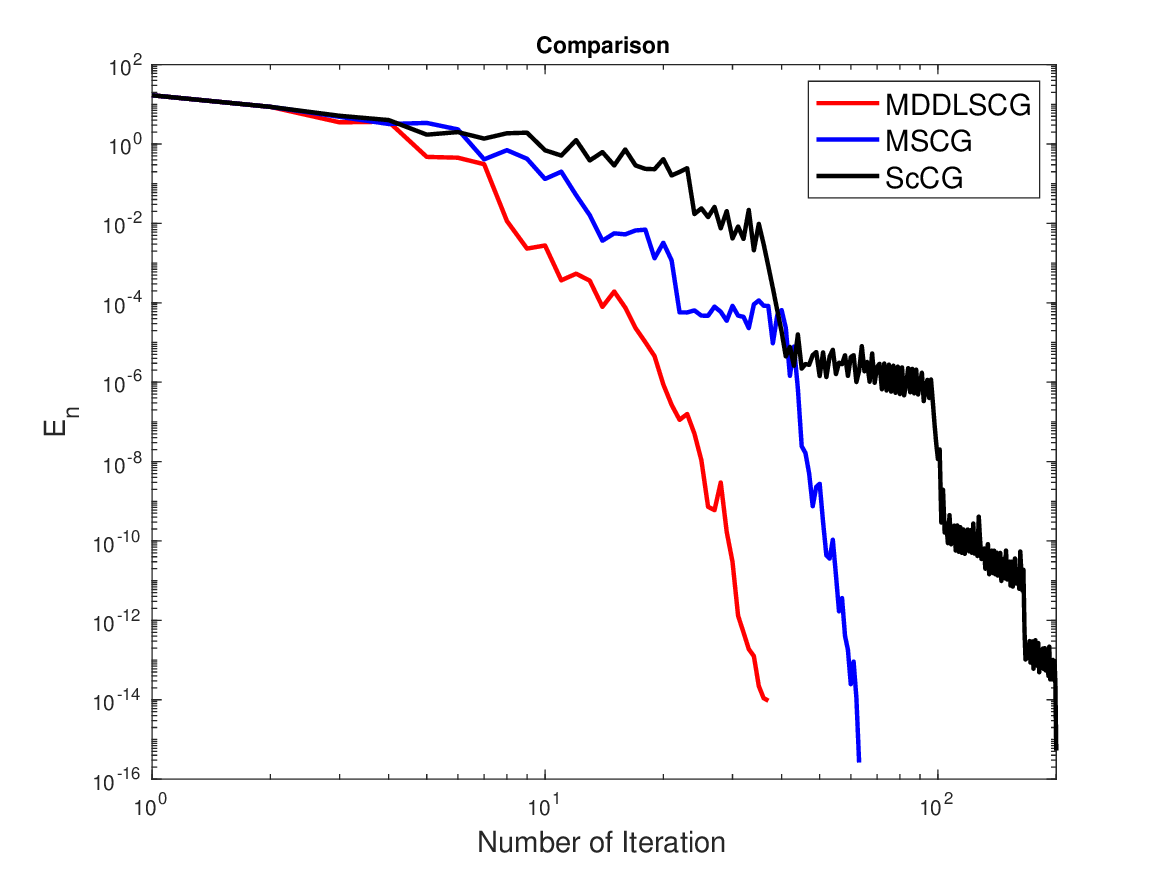}
				}
				\subfigure[Behavior of iterations paths for $\sigma =0.2$]{
					\includegraphics[width=65mm]{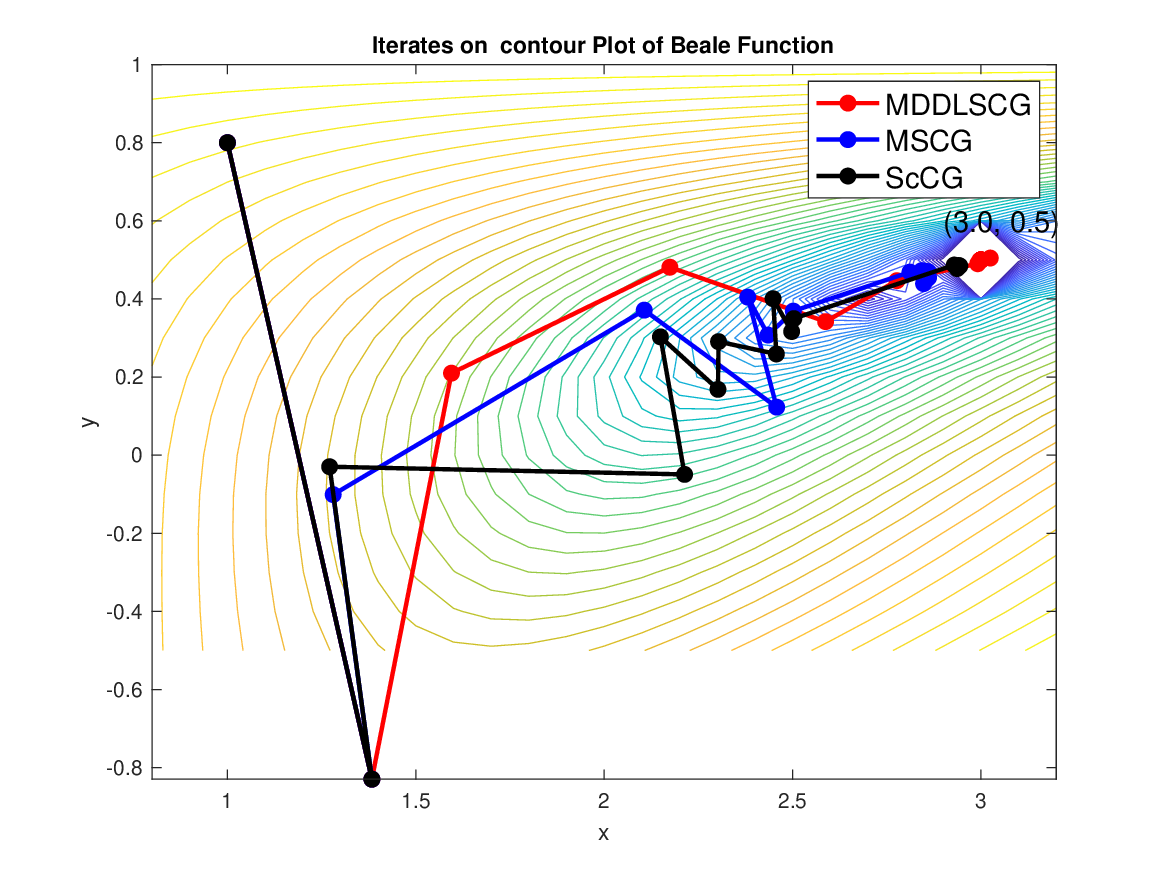}
				}%
				\subfigure[Behavior of $\|\nabla f(x_n)\|_\infty$ for $\sigma =0.2$]{
					\includegraphics[width=65mm]{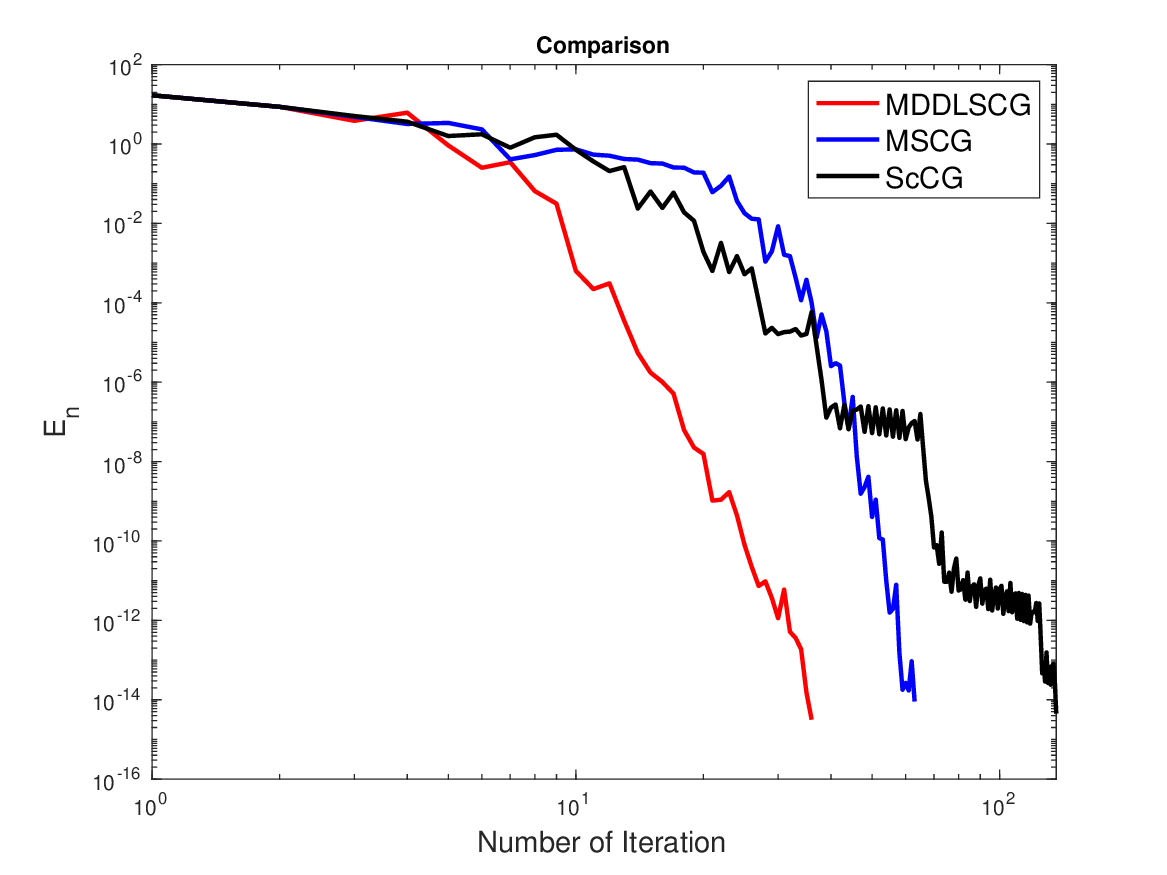}
				}
				\subfigure[Behavior of iterations paths for $\sigma =0.6$]{
					\includegraphics[width=65mm]{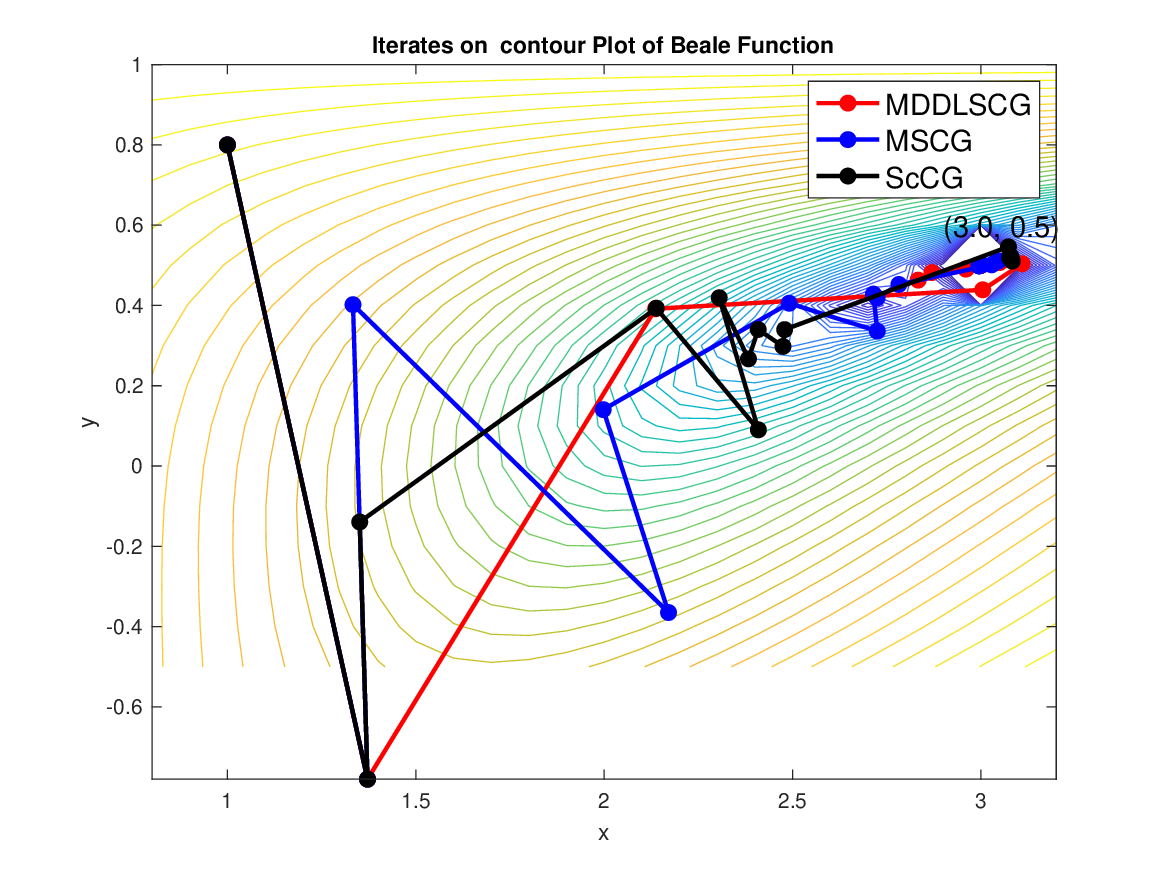}
				}%
				\subfigure[Behavior of $\|\nabla f(x_n)\|_\infty$ for $\sigma =0.6$]{
					\includegraphics[width=65mm]{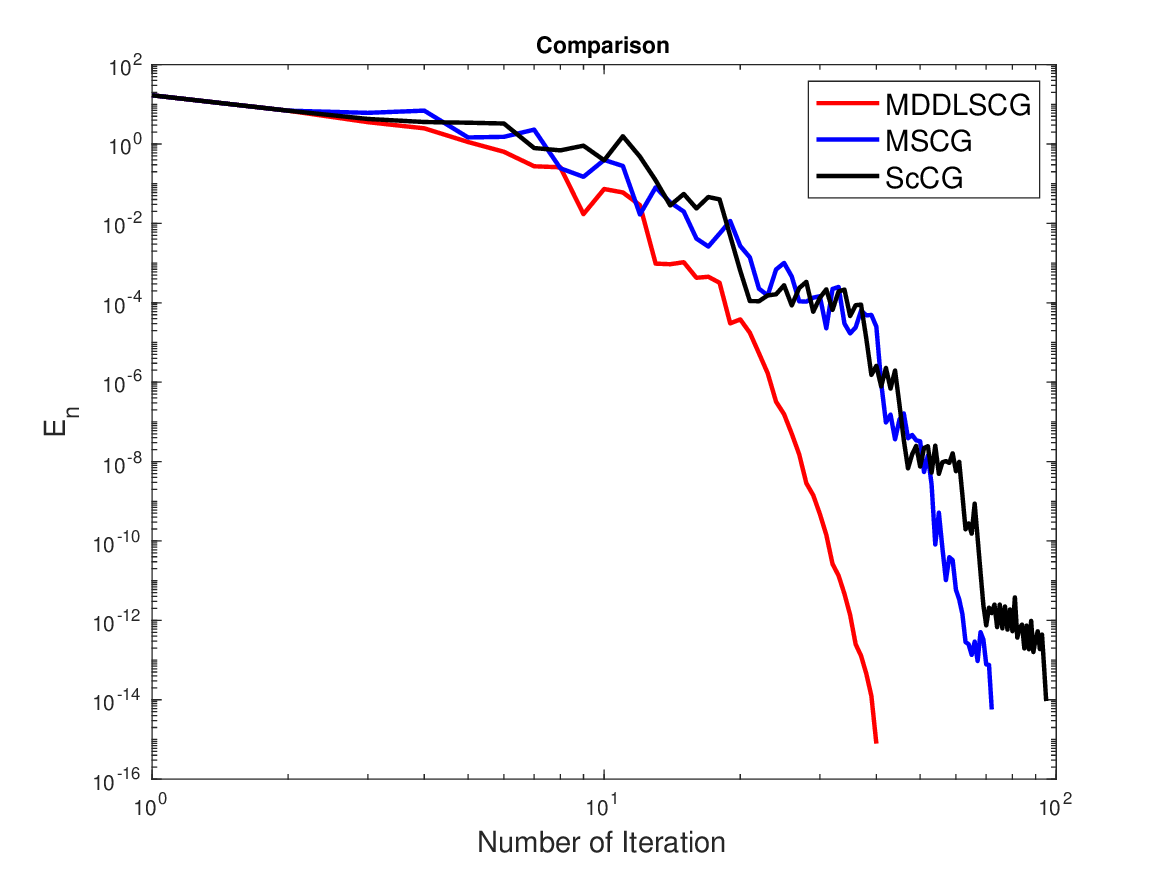}
				}
				\caption{Comparison of the iterations and execution time for different $\sigma$ for MDDLSCG, MSCG, ScCG for the Beale function }
				\label{bealefig1}
			\end{figure}

				%
				%
				%
				%
			
			\begin{figure}[htbp]
				\centering
				\subfigure[Behavior of iterations for  $\sigma$]{
					\includegraphics[width=65mm]{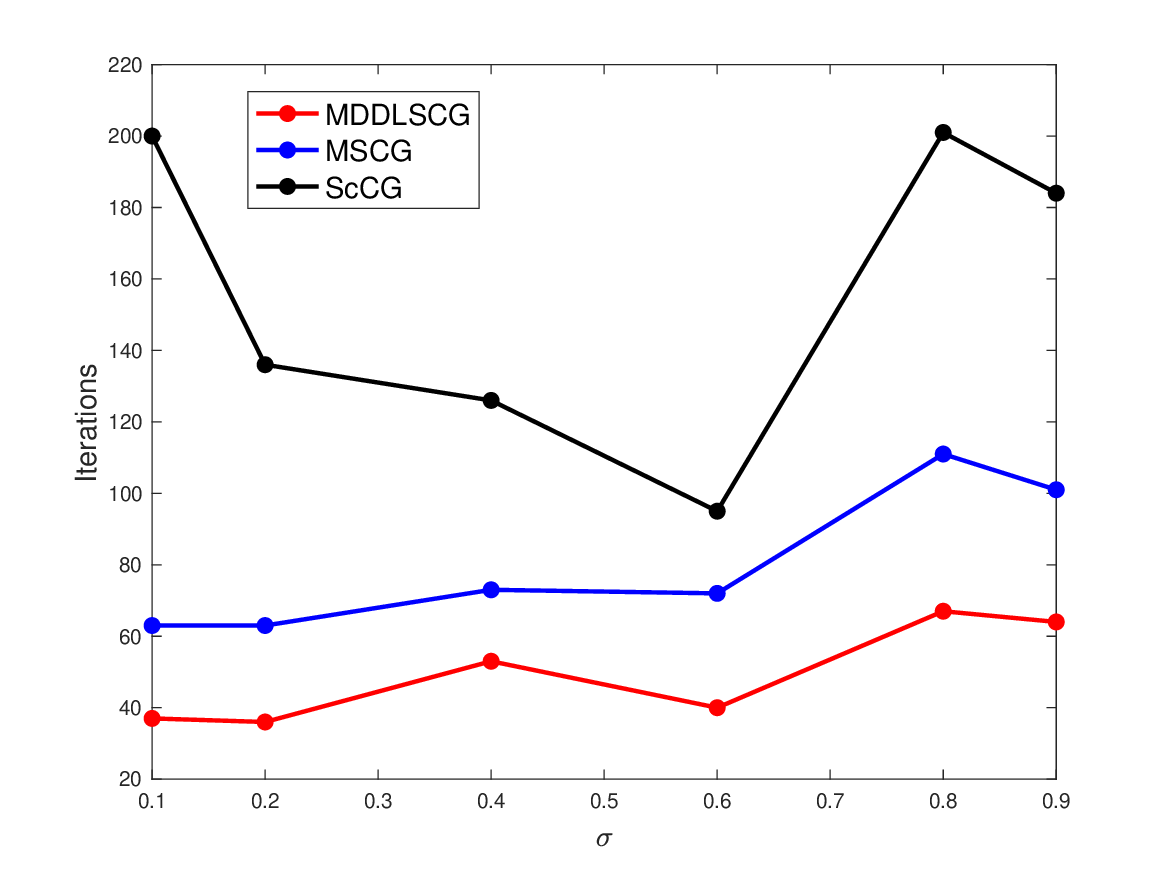}
				}%
				\subfigure[Behavior of execution time for  $\sigma$]{
					\includegraphics[width=65mm]{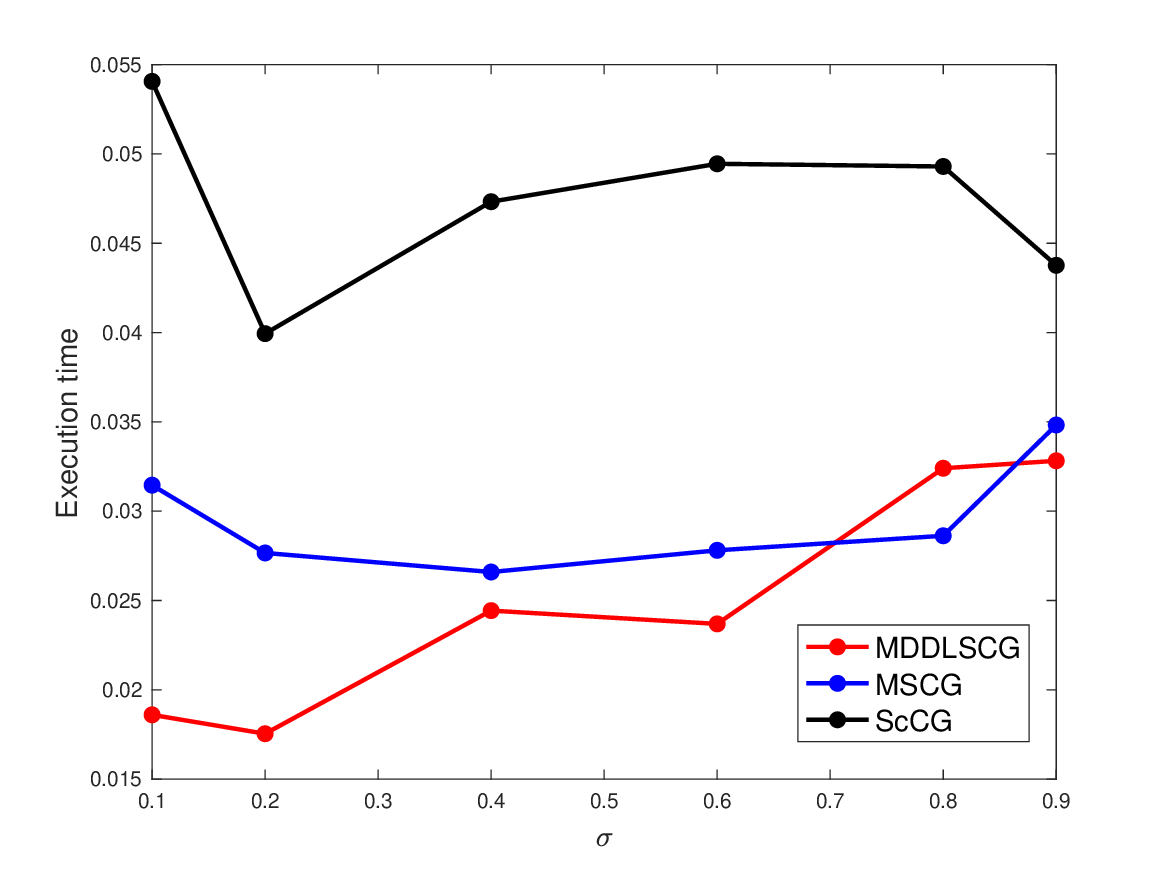}
				}
				\caption{Comparison of the iterations and execution time for different $\sigma$ for MDDLSCG, MSCG, ScCG for the Beale function }
				\label{bealefig2}
			\end{figure}  
			\ \\
			\textbf{Almost Perturbed Quadratic function:}	Consider the Almost Perturbed Quadratic (APQ) function $f:\mathbb{R}^m \to \mathbb{R}$ defined by 
			\begin{equation*}
				f(x) = \frac{1}{2} x^T A x + b^T x \text{ for all } x\in \mathbb{R}^m,
			\end{equation*}
			where $A\in \mathbb{R}^{m\times m}$ is a symmetric positive definite matrix ans $b\in \mathbb{R}^m$ is a vector of perturbations. This function has a global minimum at $x=-A^{-1}b$. Note that the gradient of $f$ is given by 
			\begin{equation*}
				\nabla f(x)=Ax+b \text{ for all } x\in \mathbb{R}^m.
			\end{equation*}

			For the numerical experiments, we have considered three cases, each characterized by dimensions of $1000$, $100$ and $25$. In the cases with dimensions $1000$ and $100$, the matrix \( A \) and vector \( b \) are randomly generated using MATLAB's built-in function \texttt{rand} and the initial point is randomly generated using MATLAB's \texttt{rand} function.
			
			In the case of dimension $25$, we take the initial point 
			\begin{align*}
				x_0 &= (x_{0,1},\ldots, x_{0,i}, \ldots,x_{0,25}), \text{ where } x_{0,i} = 1 \text{ for all } i = 1, 2,\ldots, 25 \text{ and} \\
				a&=(2,4,3,6,4,12,8,10,5,8,6,11,4,8,3,6,2,12,6,8,4,10,5,2,6),\\
				b&=(-4,-52,-18,-60,-24,-72,-72,-50,-55,-80,-78,-22,-12,\\
				& \quad -80,-27,-42,-24,-120,-30,-120,-24,-100,-60,-20,-66),\\
				A&=\text{diag}(a).
			\end{align*}
			

			In all cases, the parameters are set as follows: \( \delta = 0.01 \) and \( \sigma = 0.1 \) for all algorithms, while for MSCG and MDDLSCG, we use \( \eta = 0.001 \), \( \tau = 10 \), \( r = 1 \), and \( \nu = 0.001 \). Additionally, for the MDDLSCG algorithm, we set \( p = 0.4 \) and \( q = 0.2 \). 
			
			The iterations are terminated once they satisfy the stopping criterion
			\(
			E_n = \|\nabla f(x_n)\|_\infty < 10^{-6}
			\)
			is satisfied, where $x_n=(x_{n,1},\ldots,x_{n,25})$.
			
			The performance of the algorithms MDDLSCG, MSCG, and ScCG for the APQ function is summarized in Table \ref{Tab2} for all three cases. To provide a clearer visualization of the experiment, we have plotted the graph of the number of iterations versus \( E_n \) in Figure \ref{APQgraph}. 
			From Table \ref{Tab2} and Figure \ref{APQgraph}, it is evident that our proposed method, MDDLSCG, requires fewer iterations and less computational time to satisfy the desired stopping criterion across all cases.
			
			Figure \ref{APQtrack} illustrates how the algorithms MDDLSCG, MSCG, and ScCG approximate the desired solution \( x^* = -A^{-1}b \) at the 5th and 10th iterations in the case where the problem dimension is \( 25 \). 
			
			Furthermore, Figure \ref{APQprogress} depicts the convergence behavior of the algorithms as the number of iterations increases up to the 50th iteration for the same dimension. It is observed that MDDLSCG demonstrates superior convergence characteristics compared to MSCG and ScCG.

			%
			%

			\begin{table}[htbp]
				\centering
				\begin{tabular}{|ccccc|}
					
					\hline
					Dimension & Method & itr& Tcpu& $Er_n$\\
					\hline
					25 & MDDLSCG&52&3.088055e-02 &5.691882e-07\\
					&MSCG&59 &3.947327e-02&6.931169e-07\\
					& ScCG& 65&4.730704e-02  &4.514777e-07\\
					\hline
					100& MDDLSCG&102&  5.552326e-02&8.162825e-07\\
					&MSCG&121 &6.512598e-02 &9.411660e-07\\
					& ScCG&134&6.823947e-02&8.609522e-07\\
					\hline
					1000& MDDLSCG&132& 3.481332e-02  &9.291368e-07\\
					&MSCG&146 &3.522395e-02&9.484382e-07\\
					& ScCG&171&4.053286e-02 &8.747078e-07\\
					\hline
				\end{tabular}
				\\
				\caption{Comparison table of MDDLSCG,  MSCG and ScCG for the APQ
					function }
				\label{Tab2}
			\end{table}

			\begin{figure}[htbp]
				\centering
				\subfigure[Dimension 1000]{
					\includegraphics[width=65mm]{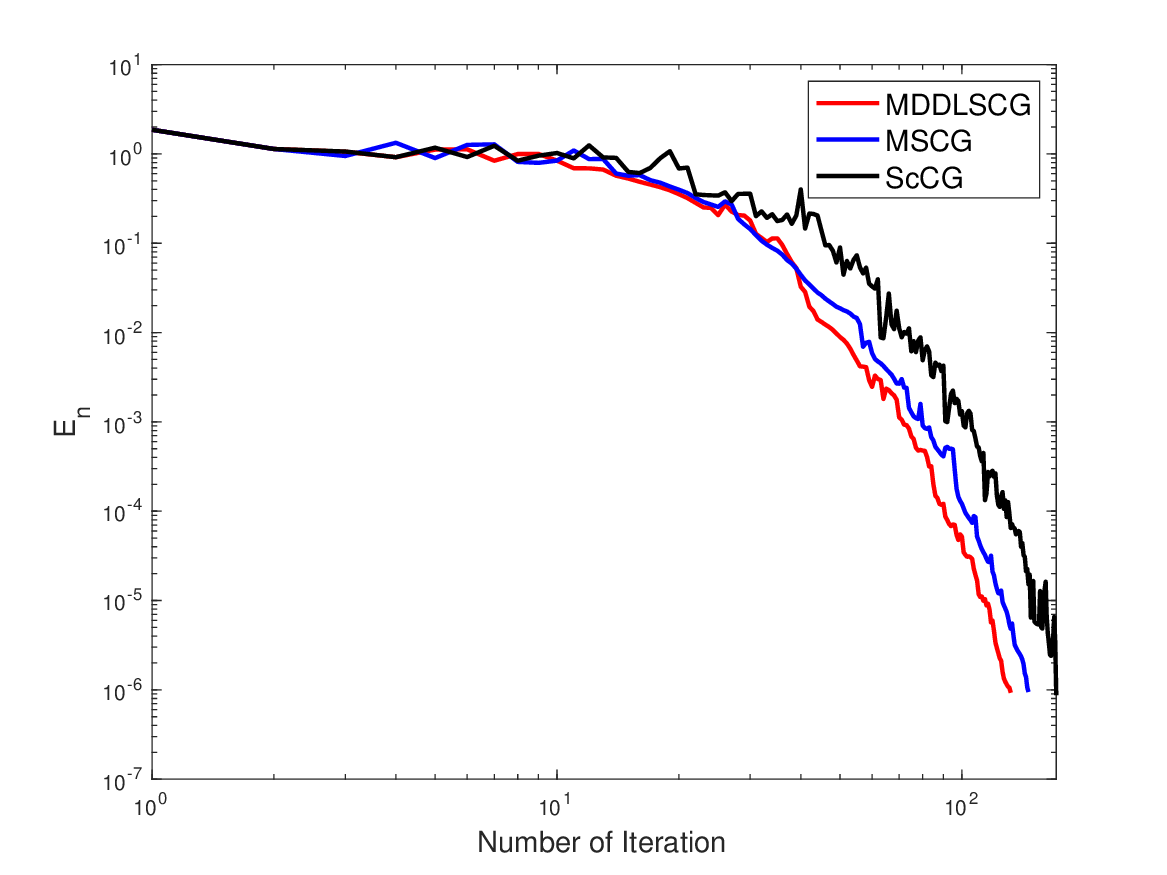}
				}%
				\subfigure[Dimension 100]{
					\includegraphics[width=65mm]{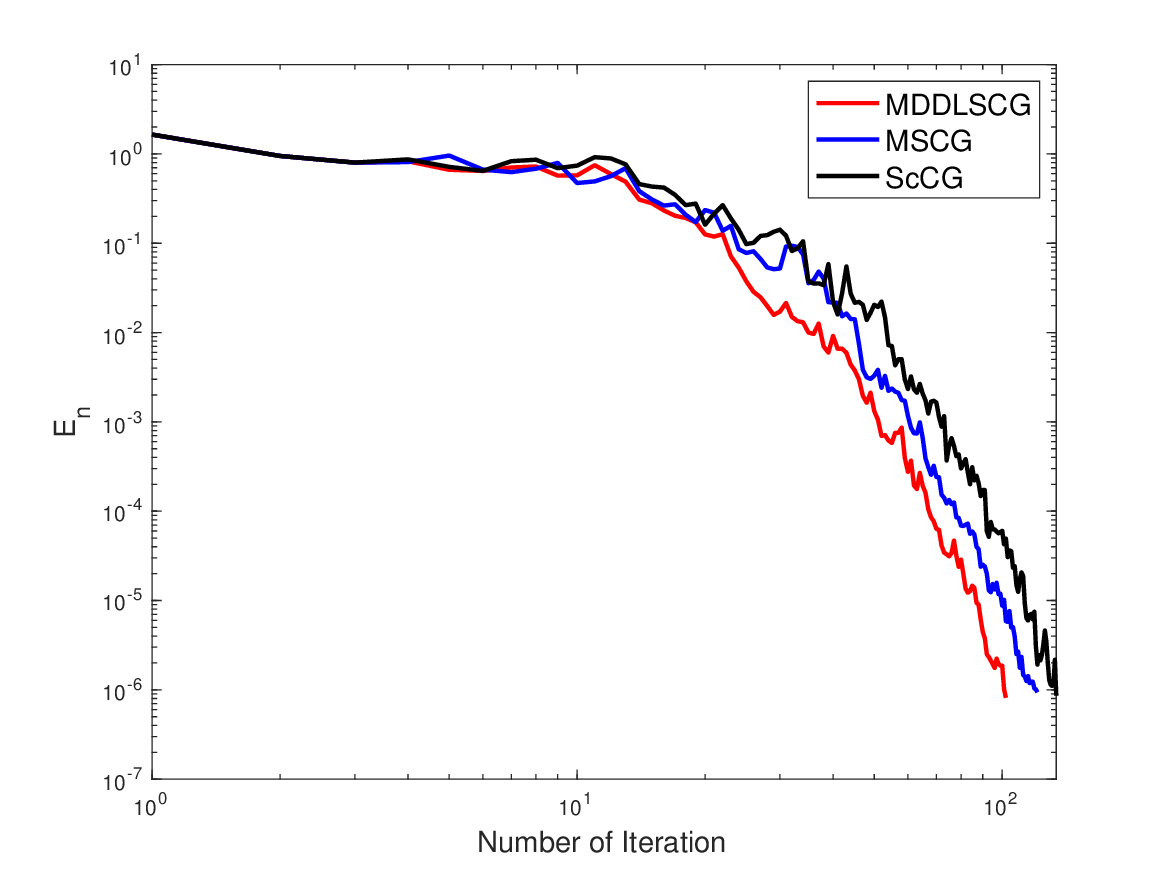}
				}
				\subfigure[Dimension 25]{
					\includegraphics[width=65mm]{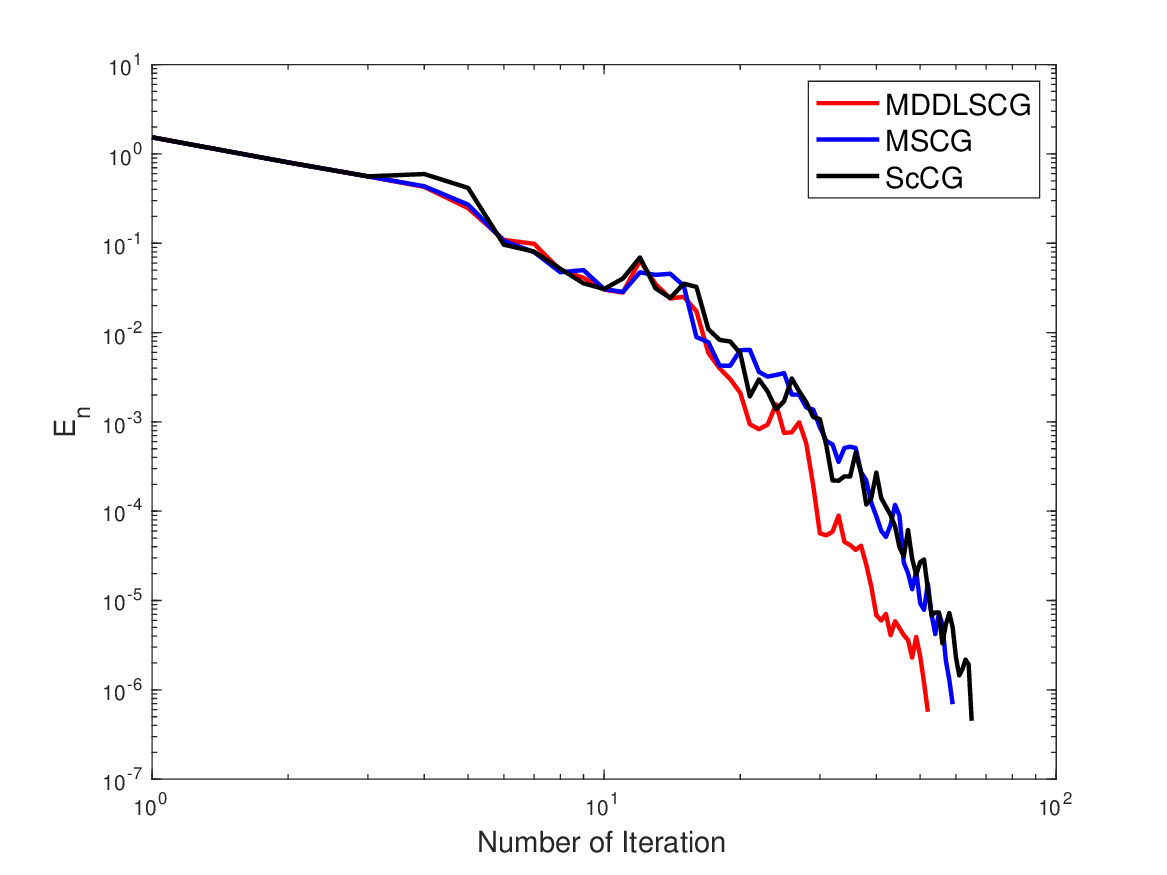}
				}
				\caption{Convergence behavior of $\|\nabla f(x_n)\|_\infty$ for MDDLSCG, MSCG, ScCG for the APQ function}
				\label{APQgraph}
			\end{figure}   
			
			\begin{figure}[htbp]
				\centering
				\subfigure[5th iteration]{
					\includegraphics[width=60mm]{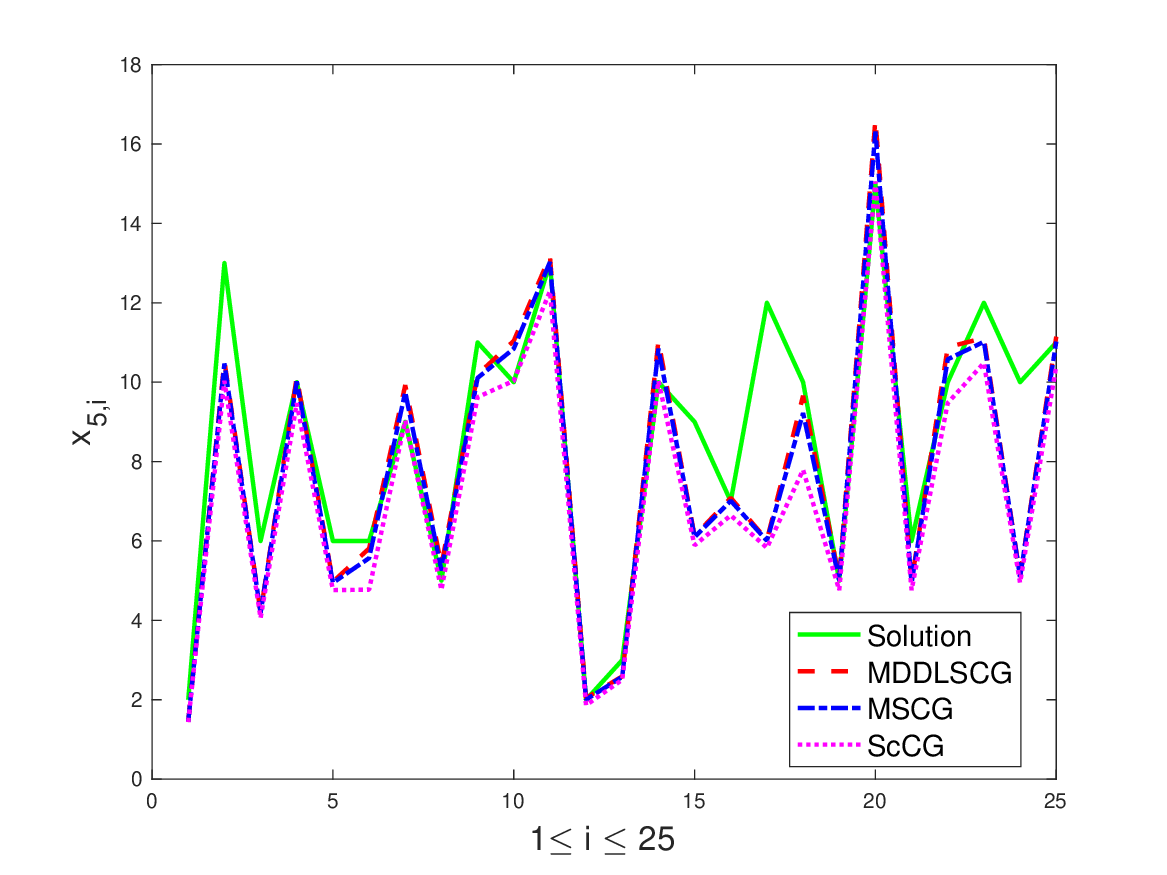}
				}%
				\subfigure[10th iterations]{
					\includegraphics[width=60mm]{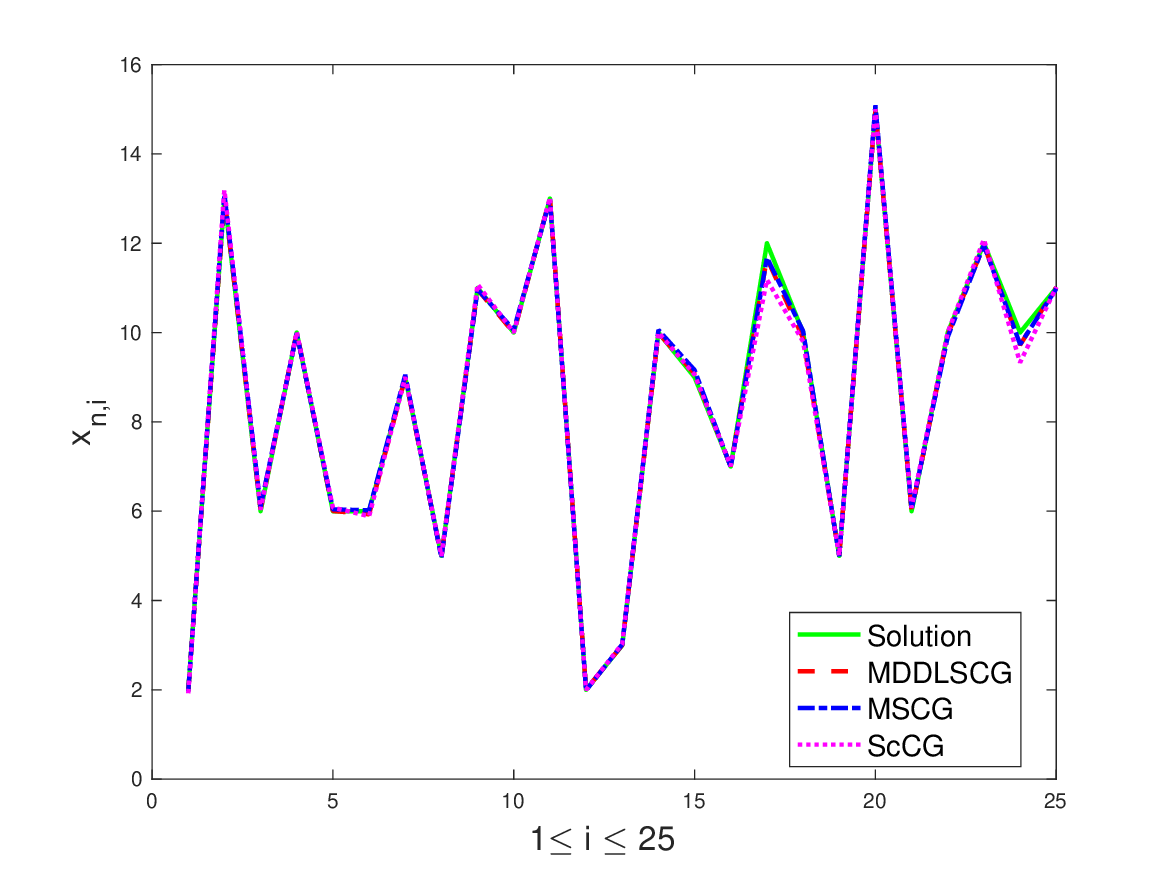}
				}
				\caption{Tracking behaviour of MDDLSCG, MSCG, ScCG for the APQ function at various iterations}
				\label{APQtrack}
			\end{figure}

			\begin{figure}[htbp]
				\centering
				\subfigure[MDDLSCG Method]{
					\includegraphics[width=60mm]{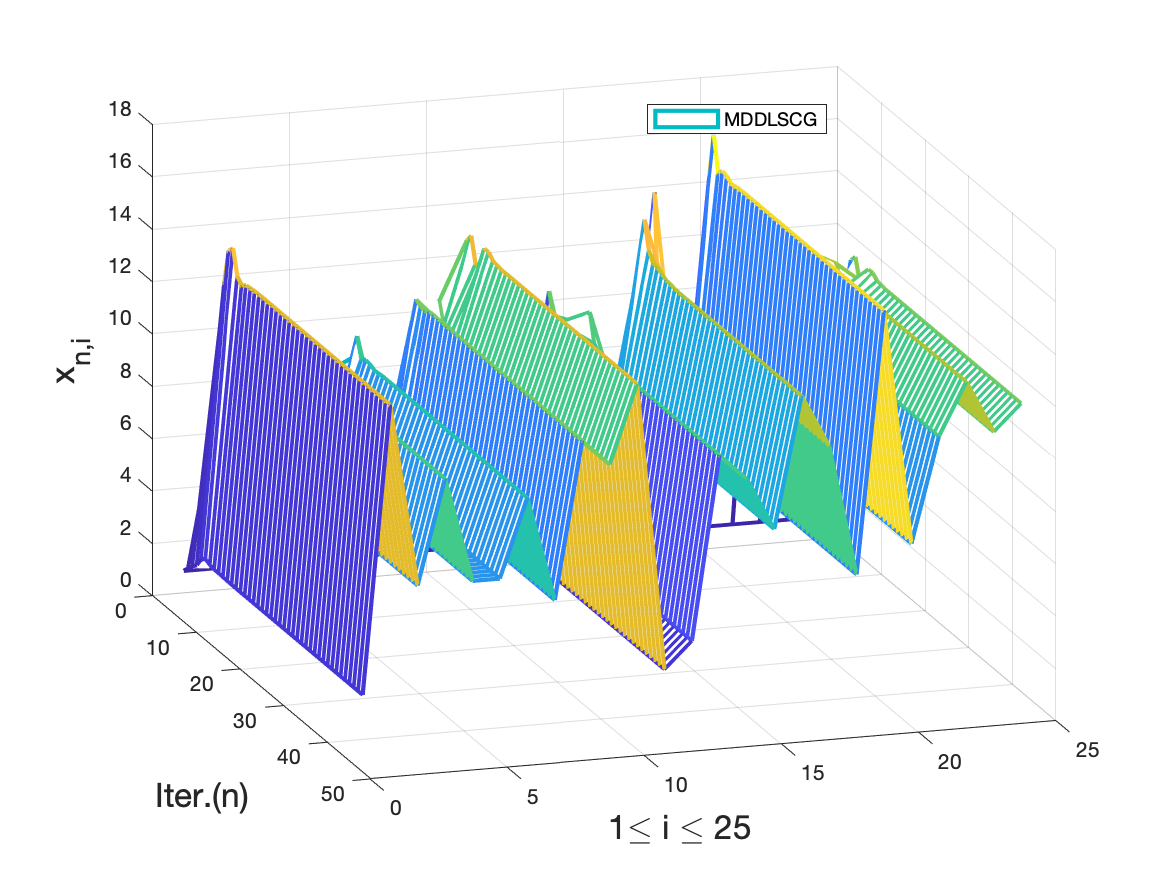}
				}%
				\subfigure[MSCG Method]{
					\includegraphics[width=60mm]{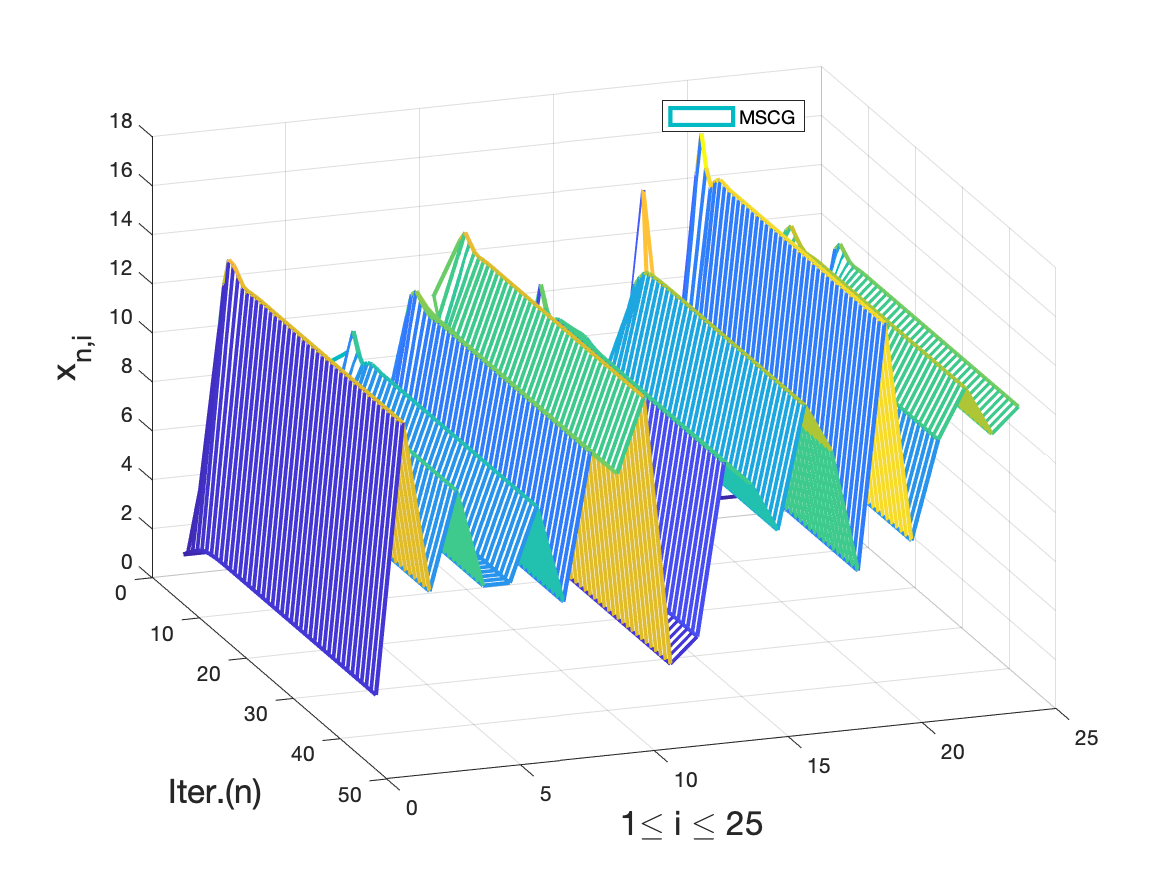}
				}
				\subfigure[ScCG Method]{
					\includegraphics[width=60mm]{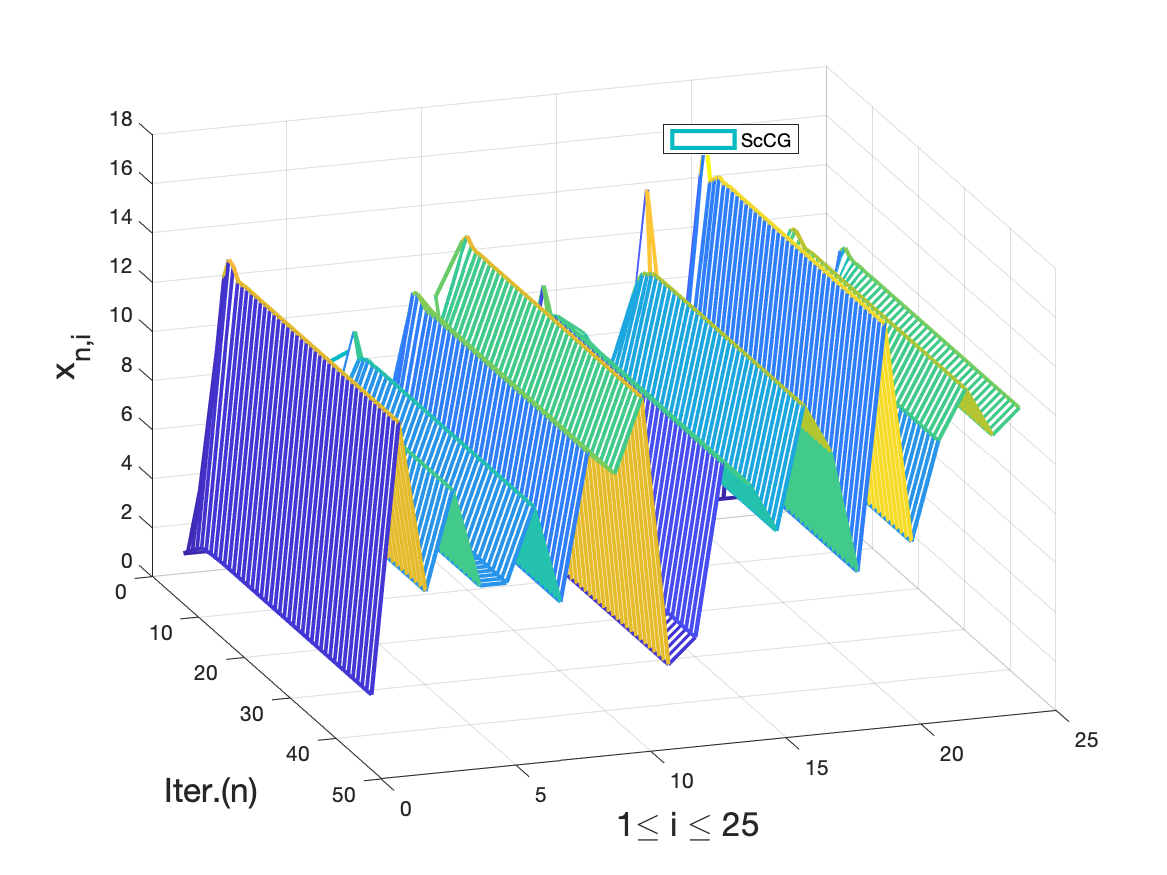}
				}
				\caption{Graphical progression of  MDDLSCG, MSCG, ScCG upto 50th iterations for the APQ function}
				\label{APQprogress}
			\end{figure}

			\section{Application to compressed sensing}
			
%
			Compressed sensing is a widely recognized signal processing approach that enables the efficient acquisition and reconstruction of signals with a sparse structure, allowing these signals to be stored in a compact and memoryless form \cite{Bruckstein2009}. Its practical significance has been demonstrated in numerous fields, including machine learning, compressive imaging, radar systems, wireless sensor networks, medical and astronomical signal processing, and video compression, as discussed in \cite{Aminifard2022}.
			
			The compressed sensing problem involves finding a sparse solution to a highly underdetermined system, represented as \(Ax = b\), where \(A \in \mathbb{R}^{m \times n}\) (with \(m \ll n\)) and \(b \in \mathbb{R}^m\) \cite{Bruckstein2009}. A common approach is to solve the following optimization problem:
			\[
			f(x) = \frac{1}{2} \|Ax - b\|^2 + \mu \phi(x),
			\]
			where \(\phi: \mathbb{R}^n \to \mathbb{R}\) is a penalty function, and \(\mu > 0\) is a parameter that regulates the trade-off between sparsity and reconstruction fidelity. A popular choice for \(\phi\) is \(\phi(x) = \|x\|_1\), leading to the well-known Basis Pursuit Denoising (BPD) problem, which has been extensively studied in the literature; see \cite{Esmaeili2019} and references therein. However, solving BPD is difficult because the \(\ell_1\)-penalty term is non-smooth. To overcome this difficulty, Zhu et al. \cite{Zhu2013} introduced a relaxed form of BPD based on Nesterov's smoothing technique \cite{Nesterov2005Siam}, leading to the following formulation:
			\[
			f_\lambda(x) = \sum_{i=1}^n \psi_\lambda (|x_i|),
			\]
			where \(x_i\) represents the \(i\)-th component of the vector \(x\), and \(\psi_\lambda\) is defined by
			\[
			\psi_\lambda(|x_i|) = \max_{\gamma \in [0,1]}\left\{\gamma|x_i| - \frac{1}{2}\lambda \gamma^2\right\},
			\]
			which simplifies to:
			\[
			\psi_\lambda(|x_i|) = \begin{cases}
				\frac{|x_i|^2}{2\lambda}, & \text{if } |x_i| < \lambda,\\
				|x_i| - \frac{\lambda}{2}, & \text{if } |x_i| \geq \lambda,
			\end{cases}
			\]
			where $\lambda>0$. Function \(f_\lambda\) is known as the Huber function \cite{Rangarajan}, and its gradient \(\nabla f_\lambda\) is  \(\frac{1}{\lambda}\) -Lipschitz continuous \cite{Becker2011}.  The explicit form of the Huber function is given by
			\[
			(\nabla f_\lambda)_i = \begin{cases}
				\frac{x_i}{\lambda}, & \text{if } |x_i| < \lambda,\\
				\text{sgn}(x_i), & \text{otherwise}.
			\end{cases}
			\]
			where  \(\text{sgn}(\cdot)\) denotes the signum function.
			
			Therefore, instead of solving the original problem for \( f(x) = \frac{1}{2} \|Ax - b\|^2 + \mu \|x\|_1 \), we solve the smoothed version:
			
			\[
			\min_{x \in \mathbb{R}^n} f(x) = \frac{1}{2} \|Ax - b\|^2 + \mu f_\lambda(x).
			\]
			
			Clearly, the problem above is an unconstrained smooth convex optimization problem, and its gradient is given by:
			
			\[
			\nabla f(x) = \lambda \nabla f_\lambda(x) + A^T(Ax - b).
			\]

			For the numerical experiments, we aim to reconstruct a sparse signal of length \( n \) from \( m \)-length observations (\( m \ll n \)). The parameters for the   MDDLSCG, MSCG and ScCG algorithms are set as follows: \\
			\(\mu = \max\{2^{-7}, 0.001\|A^T b\|_\infty\}\), \(\lambda = \min\{0.001, 0.048\|A^T b\|_\infty\}\), \(\sigma = 0.1\), 
			\(\delta = 0.01\), \(\eta=0.001\), \(\tau=10\), \(r=1\) and \(\nu=0.001\).

			For MDDLSCG, we choose $p=0.4$ and $q=0.2$. The Gaussian matrix \( A \) is generated with the MATLAB commands \texttt{randn(m, n)}  and the noise vector \( w \)  is generated with \( 0.1 \times \texttt{randn(m, 1)} \). The measurement vector \( b \) is computed by
			\(
			b = Ax + w.
			\)
			The regularization parameter is chosen as \( \mu = 0.001 \|A^T b\|_\infty \). Each methods begin with the initial point \( x_0 = A^T b \). Let \( x \) represent the true signal and \( x^* \) denote the recovered one. The termination criterion for the experiment is the mean-squared error(MSE):
			\[\text{MSE}=\frac{1}{n}\sum_{i=1}^n (x_i-x_i^*)^2  \leq 10^{-5}.\]
			To assess the reconstruction quality, we also compute the relative error (RelErr) of the restored signal, defined by
			\[	\text{RelErr} = \frac{\|x^* - x\|}{\|x\|}.\]
			
			Let \( k \) represent the number of nonzero components in the original signal. To ensure a fair comparison, we test two cases: $(m, n, k) = (128, 512, 16)$, $(256, 1024, 32)$.

			The numerical results are summarized in Table \ref{TableIner1}, where we list, for each case, the average iteration count (Itr), the average CPU time (Tcpu), the mean squared error (MSE), and the relative error (RelErr). The graph depicting the relationship between relative error and the number of iterations is shown in Figures \ref{stpp512}(a) and \ref{stopp}(a), while the graph illustrating MSE versus the number of iterations is presented in Figures \ref{stpp512}(b) and \ref{stopp}(b). In addition, the recovery of the signal is visualized in Figures \ref{signal512} and \ref{signal}.
			
			From Table \ref{TableIner1}, Figures \ref{stpp512} and \ref{stopp}, it is clear that the MDDLSCG algorithm surpasses the MSCG and ScCG algorithms in performance across all cases.

			\begin{figure}[htbp]
				\centering
				\subfigure[relative error]{
					\includegraphics[width=65mm]{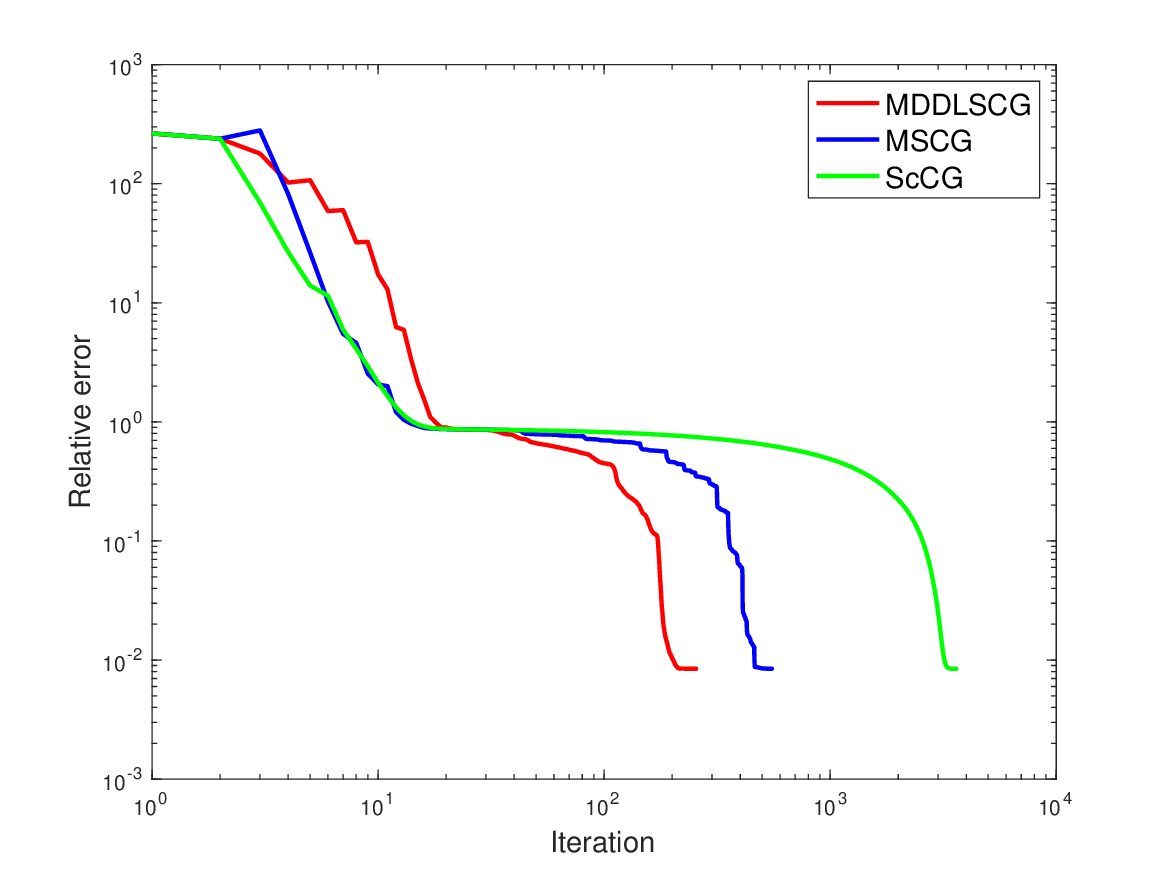}
				}%
				\subfigure[stopping criterion]{
					\includegraphics[width=65mm]{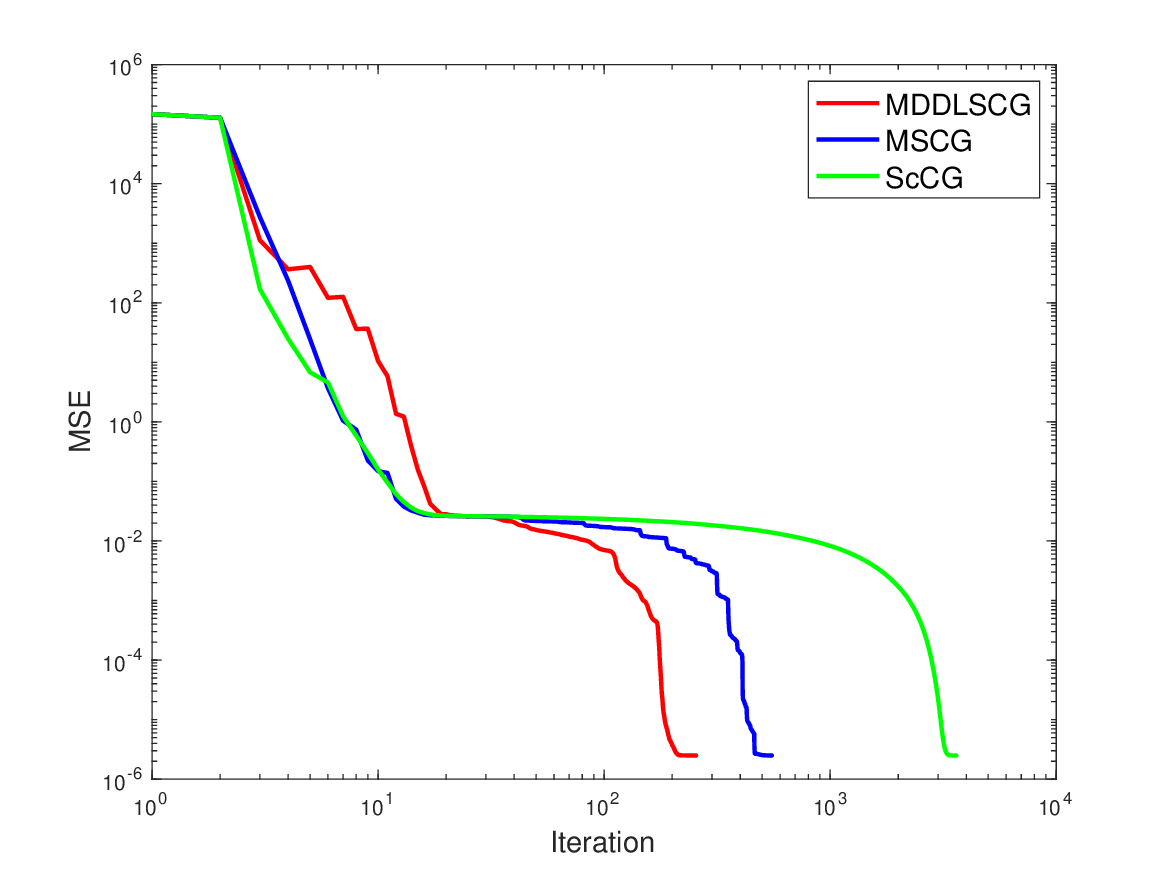}
				}
				\caption{Convergence behavior and relative error plots, demonstrating the performance of the methods MDDLSCG, MSCG and ScCG for $(m, n, k) = (128, 512, 16)$}
				\label{stpp512}
			\end{figure}

			\begin{figure}[htbp]
				\centering
				\includegraphics[width=1\linewidth]{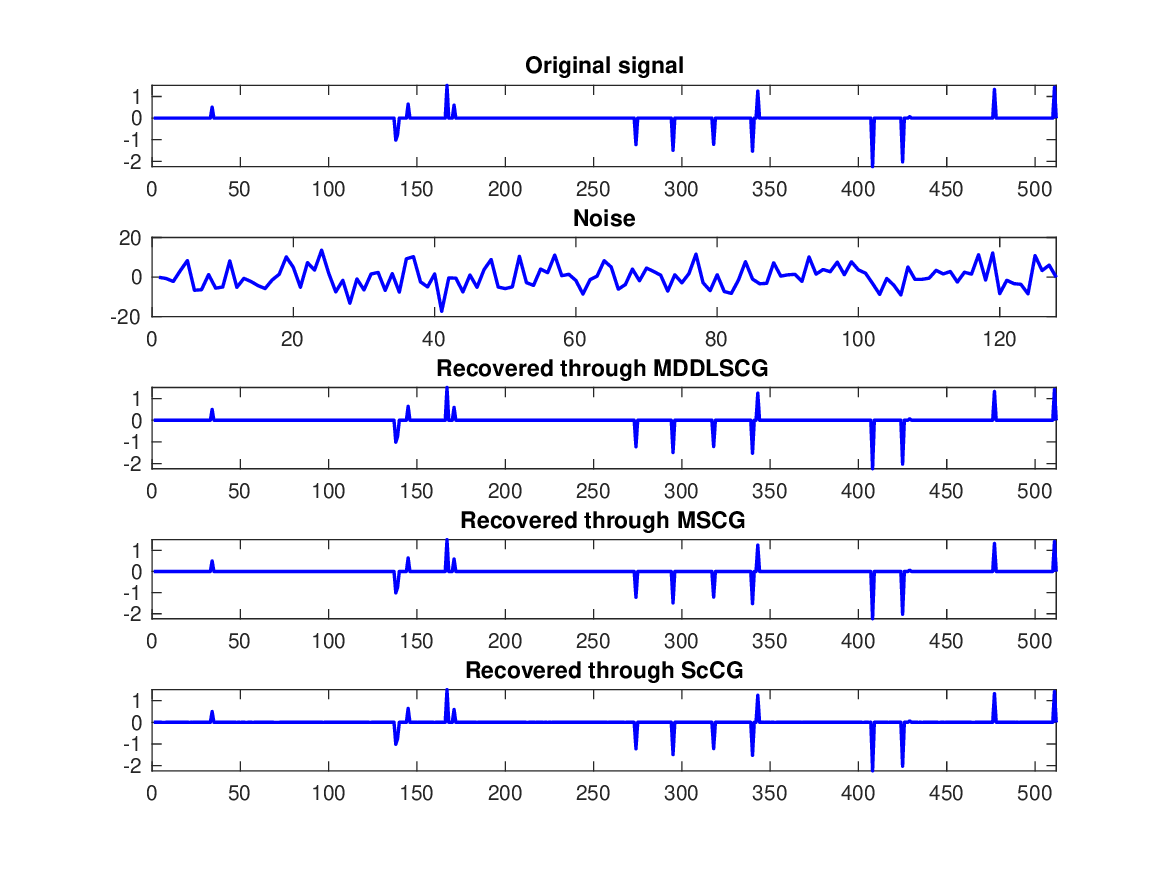}
				\caption{Signal recovery for data $(m, n, k) = (128, 512, 16)$}
				\label{signal512}
			\end{figure}

			\begin{figure}[htbp]
				\centering
				\subfigure[Relative error]{
					\includegraphics[width=65mm]{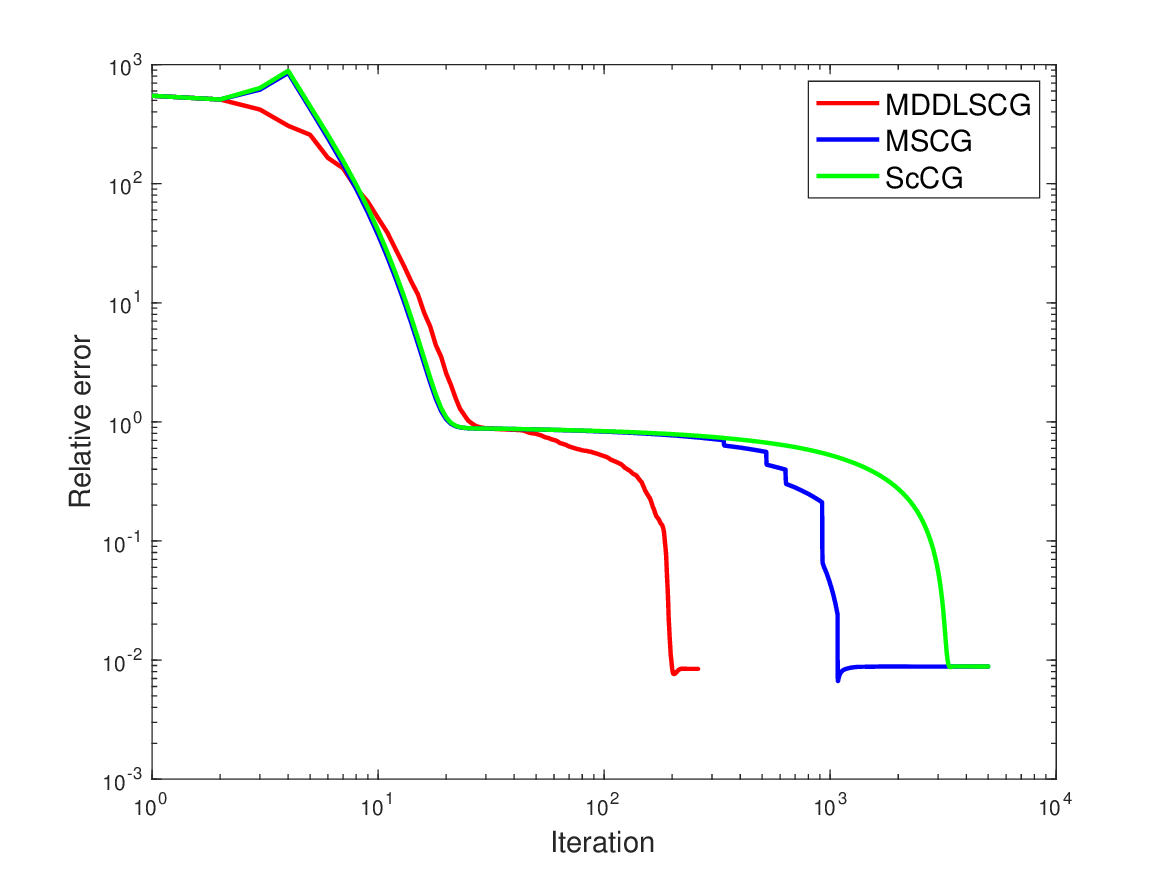}
				}%
				\subfigure[Stopping Criterion]{
					\includegraphics[width=65mm]{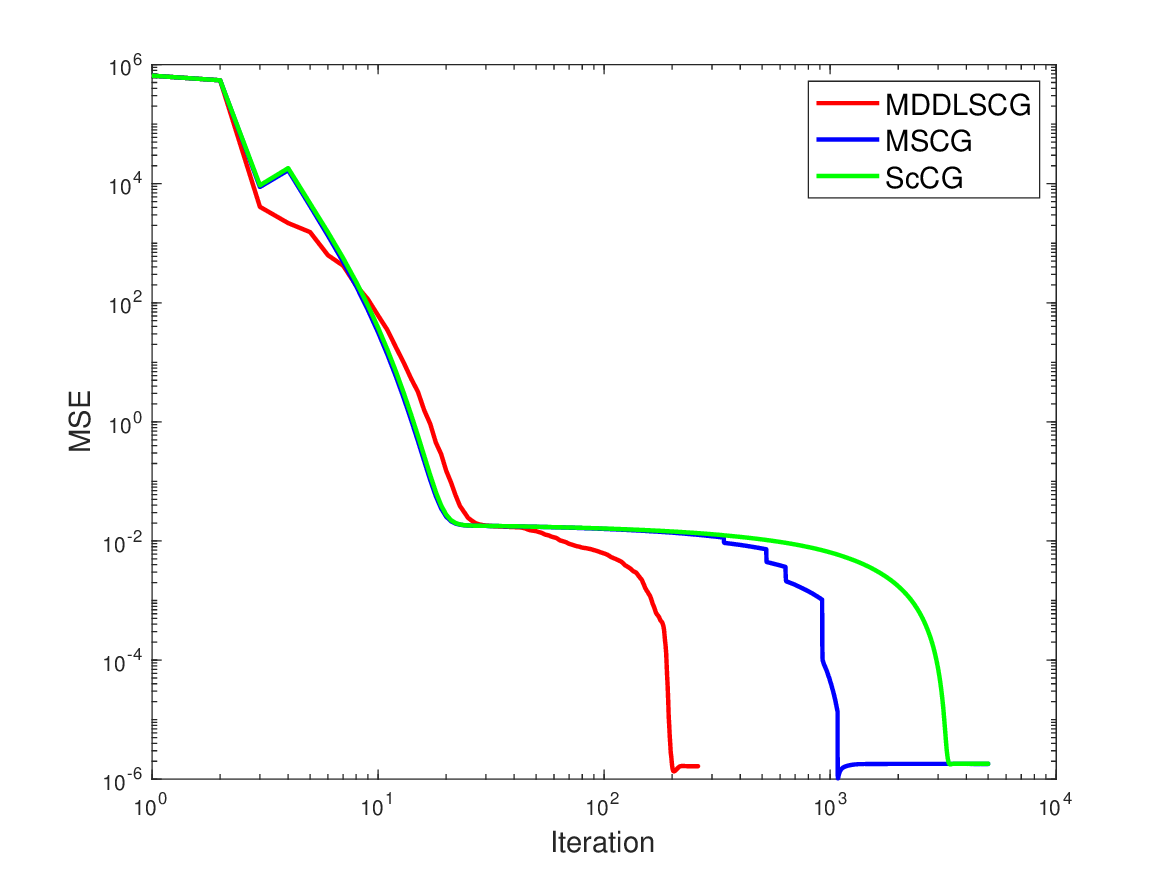}
				}
				\caption{Convergence behavior and relative error plots, demonstrating the performance of the methods MDDLSCG, MSCG and ScCG for $(m, n, k) = (256, 1024, 32)$}
				\label{stopp}
			\end{figure}

			\begin{figure}[htbp]
				\centering
				\includegraphics[width=1\linewidth]{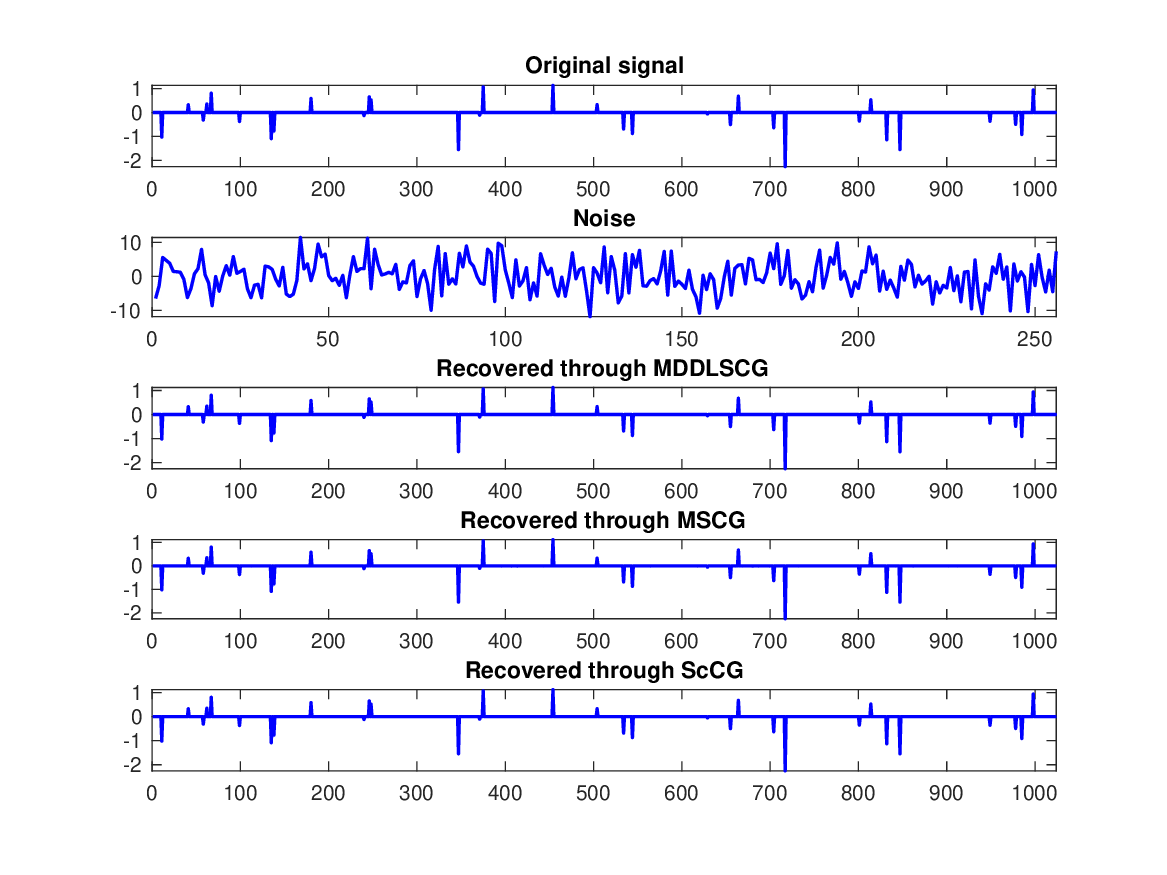}
				\caption{Signal recovery for data $(m, n, k) = (256, 1024, 32)$}
				\label{signal}
			\end{figure}
			
			\begin{table}[htbp]
				\centering
				\begin{adjustbox}{width=1\textwidth}
					\begin{tabular}{|cccccc|}
						\hline
						Dim$(m,n,k)$ & Method & Itr& Tcpu& MSE &RelErr\\
						\hline
						(128,512,16)&MDDLSCG&272 &1.186382e-01 &2.033486e-06 &7.487872e-03\\
						&MSCG &493& 1.606987e-01&2.033495e-06 &7.488087e-03\\
						& ScCG&3653& 1.135658e+00&2.033588e-06 &7.712389e-03\\
						\hline
						(256,1024,32)&MDDLSCG&291& 2.163503e-01&2.128598e-06 &9.721181e-03\\
						& MSCG&1405& 8.339790e-01 & 2.128516e-06 &9.720993e-03\\
						& ScCG&3002& 1.746555e+00 & 2.229608e-04 &9.949161e-02\\
						\hline
					\end{tabular}
				\end{adjustbox}
				\caption{Comparison table for methods MDDLSCG, MSCG and ScCG}
				\label{TableIner1}
			\end{table}
			
			\section{Conclusions}

			We investigated a modified descent Dai-Liao spectral conjugate gradient (MDDLSCG) method that incorporates a new secant condition and quasi-Newton direction. The updated method introduces spectral parameters ensuring a sufficient descent property independent of line search techniques. Theoretical analysis establishes the global convergence of our method for general nonlinear functions under standard assumptions. Numerical experiments and an application to signal processing demonstrate the improved performance and practical effectiveness of the proposed algorithm. Additionally, we present numerical results that validate its performance and show its superior convergence compared to the MSCG method \cite{Faramarzi2019} and the ScCG method \cite{Mrad2024}, highlighting its potential in optimization and related applications.

			\section*{Declaration of competing interest}
			The authors declare that they have no known competing financial interests or personal relationships that could have appeared to influence the work reported in this paper.
			
			\section*{Funding}
			The authors declare that no funds, grants, or other support were received 
			during the preparation of this manuscript.
			
			\bibliographystyle{elsarticle-num} 
			\bibliography{MDDL}

@article{Aminifard2022,
  title={Modified conjugate gradient method for solving sparse recovery problem with nonconvex penalty},
  author={Aminifard, Zohre and Hosseini, Alireza and Babaie-Kafaki, Saman},
  journal={Signal Processing},
  volume={193},
  pages={108424},
  year={2022},
  publisher={Elsevier}
}

@article {Andrei2007,
	AUTHOR = {Andrei, Neculai},
	TITLE = {A scaled {BFGS} preconditioned conjugate gradient algorithm
	for unconstrained optimization},
	JOURNAL = {Appl. Math. Lett.},
	FJOURNAL = {Applied Mathematics Letters. An International Journal of Rapid
	Publication},
	VOLUME = {20},
	YEAR = {2007},
	NUMBER = {6},
	PAGES = {645--650},
	ISSN = {0893-9659,1873-5452},
	MRCLASS = {90C55 (90C30 90C52)},
	MRNUMBER = {2314407},
	DOI = {10.1016/j.aml.2006.06.015},
}

@article {Andrei2008,
	AUTHOR = {Andrei, Neculai},
	TITLE = {Another hybrid conjugate gradient algorithm for unconstrained
	optimization},
	JOURNAL = {Numer. Algorithms},
	FJOURNAL = {Numerical Algorithms},
	VOLUME = {47},
	YEAR = {2008},
	NUMBER = {2},
	PAGES = {143--156},
	ISSN = {1017-1398,1572-9265},
	MRCLASS = {90C52 (65K10)},
	MRNUMBER = {2383261},
	DOI = {10.1007/s11075-007-9152-9},
}

@article {Andrei2010EJOR,
	AUTHOR = {Andrei, Neculai},
	TITLE = {Accelerated scaled memoryless {BFGS} preconditioned conjugate
	gradient algorithm for unconstrained optimization},
	JOURNAL = {European J. Oper. Res.},
	FJOURNAL = {European Journal of Operational Research},
	VOLUME = {204},
	YEAR = {2010},
	NUMBER = {3},
	PAGES = {410--420},
	ISSN = {0377-2217,1872-6860},
	MRCLASS = {90C52 (90C30)},
	MRNUMBER = {2587869},
	DOI = {10.1016/j.ejor.2009.11.030},
}

@article {Barzilai1988,
	AUTHOR = {Barzilai, Jonathan and Borwein, Jonathan M.},
	TITLE = {Two-point step size gradient methods},
	JOURNAL = {IMA J. Numer. Anal.},
	FJOURNAL = {IMA Journal of Numerical Analysis},
	VOLUME = {8},
	YEAR = {1988},
	NUMBER = {1},
	PAGES = {141--148},
	ISSN = {0272-4979},
	MRCLASS = {65K05},
	MRNUMBER = {967848},
	DOI = {10.1093/imanum/8.1.141},
}

@article {Babaie2013MMA,
	AUTHOR = {Babaie-Kafaki, Saman and Mahdavi-Amiri, Nezam},
	TITLE = {Two modified hybrid conjugate gradient methods based on a
	hybrid secant equation},
	JOURNAL = {Math. Model. Anal.},
	FJOURNAL = {Mathematical Modelling and Analysis},
	VOLUME = {18},
	YEAR = {2013},
	NUMBER = {1},
	PAGES = {32--52},
	ISSN = {1392-6292,1648-3510},
	MRCLASS = {90C52},
	MRNUMBER = {3032464},
	DOI = {10.3846/13926292.2013.756832},
}

@article {KafakiGh2014,
	AUTHOR = {Babaie-Kafaki, Saman and Ghanbari, Reza},
	TITLE = {A descent family of {D}ai-{L}iao conjugate gradient methods},
	JOURNAL = {Optim. Methods Softw.},
	FJOURNAL = {Optimization Methods \& Software},
	VOLUME = {29},
	YEAR = {2014},
	NUMBER = {3},
	PAGES = {583--591},
	ISSN = {1055-6788,1029-4937},
	MRCLASS = {90C26 (90C52)},
	MRNUMBER = {3175504},
	MRREVIEWER = {Ning\ Ruan},
	DOI = {10.1080/10556788.2013.833199},
}

@article {Becker2011,
	AUTHOR = {Becker, Stephen and Bobin, J\'er\^ome and Cand\`es, Emmanuel
	J.},
	TITLE = {N{ESTA}: a fast and accurate first-order method for sparse
	recovery},
	JOURNAL = {SIAM J. Imaging Sci.},
	FJOURNAL = {SIAM Journal on Imaging Sciences},
	VOLUME = {4},
	YEAR = {2011},
	NUMBER = {1},
	PAGES = {1--39},
	ISSN = {1936-4954},
	MRCLASS = {90C06 (90C25 94A08)},
	MRNUMBER = {2765668},
	DOI = {10.1137/090756855},
}

@article {Bruckstein2009,
	AUTHOR = {Bruckstein, Alfred M. and Donoho, David L. and Elad, Michael},
	TITLE = {From sparse solutions of systems of equations to sparse
	modeling of signals and images},
	JOURNAL = {SIAM Rev.},
	FJOURNAL = {SIAM Review},
	VOLUME = {51},
	YEAR = {2009},
	NUMBER = {1},
	PAGES = {34--81},
	ISSN = {0036-1445,1095-7200},
	MRCLASS = {94A08 (65F50 68U10)},
	MRNUMBER = {2481111},
	MRREVIEWER = {Anhua\ Lin},
	DOI = {10.1137/060657704},
}

@article {Birgin2001,
	AUTHOR = {Birgin, E. G. and Mart\'inez, J. M.},
	TITLE = {A spectral conjugate gradient method for unconstrained
	optimization},
	JOURNAL = {Appl. Math. Optim.},
	FJOURNAL = {Applied Mathematics and Optimization},
	VOLUME = {43},
	YEAR = {2001},
	NUMBER = {2},
	PAGES = {117--128},
	ISSN = {0095-4616,1432-0606},
	MRCLASS = {90C30 (65K05 90C52)},
	MRNUMBER = {1814590},
	MRREVIEWER = {Marcos\ Raydan},
	DOI = {10.1007/s00245-001-0003-0},
}

@article {Dai1999,
	AUTHOR = {Dai, Y. H. and Yuan, Y.},
	TITLE = {A nonlinear conjugate gradient method with a strong global
	convergence property},
	JOURNAL = {SIAM J. Optim.},
	FJOURNAL = {SIAM Journal on Optimization},
	VOLUME = {10},
	YEAR = {1999},
	NUMBER = {1},
	PAGES = {177--182},
	ISSN = {1052-6234,1095-7189},
	MRCLASS = {90C30 (65K05)},
	MRNUMBER = {1740963},
	DOI = {10.1137/S1052623497318992},
}

@article {Dai2001,
	AUTHOR = {Dai, Y.-H. and Liao, L.-Z.},
	TITLE = {New conjugacy conditions and related nonlinear conjugate
	gradient methods},
	JOURNAL = {Appl. Math. Optim.},
	FJOURNAL = {Applied Mathematics and Optimization},
	VOLUME = {43},
	YEAR = {2001},
	NUMBER = {1},
	PAGES = {87--101},
	ISSN = {0095-4616,1432-0606},
	MRCLASS = {90C30 (65K05)},
	MRNUMBER = {1804396},
	MRREVIEWER = {Hiroshi\ Yabe},
	DOI = {10.1007/s002450010019},
}

@article {DaiKou2013,
	AUTHOR = {Dai, Yu-Hong and Kou, Cai-Xia},
	TITLE = {A nonlinear conjugate gradient algorithm with an optimal
	property and an improved {W}olfe line search},
	JOURNAL = {SIAM J. Optim.},
	FJOURNAL = {SIAM Journal on Optimization},
	VOLUME = {23},
	YEAR = {2013},
	NUMBER = {1},
	PAGES = {296--320},
	ISSN = {1052-6234,1095-7189},
	MRCLASS = {90C30 (90C06 90C52 90C53)},
	MRNUMBER = {3033109},
	MRREVIEWER = {Giovanni\ Fasano},
	DOI = {10.1137/100813026},
}

@article {Esmaeili2019,
	AUTHOR = {Esmaeili, Hamid and Shabani, Shima and Kimiaei, Morteza},
	TITLE = {A new generalized shrinkage conjugate gradient method for
	sparse recovery},
	JOURNAL = {Calcolo},
	FJOURNAL = {Calcolo. A Quarterly on Numerical Analysis and Theory of
	Computation},
	VOLUME = {56},
	YEAR = {2019},
	NUMBER = {1},
	PAGES = {Paper No. 1, 38},
	ISSN = {0008-0624,1126-5434},
	MRCLASS = {65K05 (90C06 90C25 90C52 94A08)},
	MRNUMBER = {3882971},
	MRREVIEWER = {Maryia\ Kurdina},
	DOI = {10.1007/s10092-018-0296-x},
}

@article {Faramarzi2019,
	AUTHOR = {Faramarzi, Parvaneh and Amini, Keyvan},
	TITLE = {A modified spectral conjugate gradient method with global
	convergence},
	JOURNAL = {J. Optim. Theory Appl.},
	FJOURNAL = {Journal of Optimization Theory and Applications},
	VOLUME = {182},
	YEAR = {2019},
	NUMBER = {2},
	PAGES = {667--690},
	ISSN = {0022-3239,1573-2878},
	MRCLASS = {90C30 (65K05 90C52)},
	MRNUMBER = {3968308},
	MRREVIEWER = {Dongyi\ Liu},
	DOI = {10.1007/s10957-019-01527-6},
}

@book{Fletcher1987,
 AUTHOR = {Fletcher, R.},
 TITLE = {Practical methods of optimization},
 SERIES = {A Wiley-Interscience Publication},
 EDITION = {Second},
 PUBLISHER = {John Wiley \& Sons, Ltd.},
 YEAR = {1987},
 PAGES = {xiv+436},
 ISBN = {0-471-91547-5},
 ADDRESS = {Chichester}
 }

@article {Fletcher,
	AUTHOR = {Fletcher, R. and Reeves, C. M.},
	TITLE = {Function minimization by conjugate gradients},
	JOURNAL = {Comput. J.},
	FJOURNAL = {The Computer Journal},
	VOLUME = {7},
	YEAR = {1964},
	PAGES = {149--154},
	ISSN = {0010-4620},
	MRCLASS = {65.30},
	MRNUMBER = {187375},
	DOI = {10.1093/comjnl/7.2.149},
}

@article {HagerZhang2005,
	AUTHOR = {Hager, William W. and Zhang, Hongchao},
	TITLE = {A new conjugate gradient method with guaranteed descent and an
	efficient line search},
	JOURNAL = {SIAM J. Optim.},
	FJOURNAL = {SIAM Journal on Optimization},
	VOLUME = {16},
	YEAR = {2005},
	NUMBER = {1},
	PAGES = {170--192},
	ISSN = {1052-6234,1095-7189},
	MRCLASS = {90C52 (90C06 90C26)},
	MRNUMBER = {2177774},
	MRREVIEWER = {Stephen\ C.\ Billups},
	DOI = {10.1137/030601880},
}

@article{HagerZhang2006,
  title={A survey of nonlinear conjugate gradient methods},
  author={Hager, William W and Zhang, Hongchao},
  journal={Pacific Journal of Optimization},
  volume={2},
  number={1},
  pages={35--58},
  year={2006}
}

@article {Hestenes1952,
	AUTHOR = {Hestenes, Magnus R. and Stiefel, Eduard},
	TITLE = {Methods of conjugate gradients for solving linear systems},
	JOURNAL = {J. Research Nat. Bur. Standards},
	FJOURNAL = {J. Research Nat. Bur. Standards},
	VOLUME = {49},
	YEAR = {1952},
	PAGES = {409--436},
	MRCLASS = {65.0X},
	MRNUMBER = {60307},
	MRREVIEWER = {A.\ S.\ Householder},
}

@article {Jian2017OMS,
	AUTHOR = {Jian, Jinbao and Chen, Qian and Jiang, Xianzhen and Zeng,
	Youfang and Yin, Jianghua},
	TITLE = {A new spectral conjugate gradient method for large-scale
	unconstrained optimization},
	JOURNAL = {Optim. Methods Softw.},
	FJOURNAL = {Optimization Methods \& Software},
	VOLUME = {32},
	YEAR = {2017},
	NUMBER = {3},
	PAGES = {503--515},
	ISSN = {1055-6788,1029-4937},
	MRCLASS = {90C53 (90C52)},
	MRNUMBER = {3630453},
	MRREVIEWER = {Jinyan\ Fan},
	DOI = {10.1080/10556788.2016.1225213},
}

@article {Kafaki2016,
	AUTHOR = {Babaie-Kafaki, Saman},
	TITLE = {On optimality of two adaptive choices for the parameter of
	{D}ai-{L}iao method},
	JOURNAL = {Optim. Lett.},
	FJOURNAL = {Optimization Letters},
	VOLUME = {10},
	YEAR = {2016},
	NUMBER = {8},
	PAGES = {1789--1797},
	ISSN = {1862-4472,1862-4480},
	MRCLASS = {90C53 (65K10 90C26)},
	MRNUMBER = {3556960},
	MRREVIEWER = {Stefano\ Fanelli},
	DOI = {10.1007/s11590-015-0965-5},
}

@article {KafakiEJOR2014,
	AUTHOR = {Babaie-Kafaki, Saman and Ghanbari, Reza},
	TITLE = {The {D}ai-{L}iao nonlinear conjugate gradient method with
	optimal parameter choices},
	JOURNAL = {European J. Oper. Res.},
	FJOURNAL = {European Journal of Operational Research},
	VOLUME = {234},
	YEAR = {2014},
	NUMBER = {3},
	PAGES = {625--630},
	ISSN = {0377-2217,1872-6860},
	MRCLASS = {90C52 (90C06)},
	MRNUMBER = {3151323},
	MRREVIEWER = {Mehiddin\ Al-Baali},
	DOI = {10.1016/j.ejor.2013.11.012},
}

@article {Li2001,
	AUTHOR = {Li, Dong-Hui and Fukushima, Masao},
	TITLE = {A modified {BFGS} method and its global convergence in
	nonconvex minimization},
	NOTE = {Nonlinear programming and variational inequalities (Kowloon,
	1998)},
	JOURNAL = {J. Comput. Appl. Math.},
	FJOURNAL = {Journal of Computational and Applied Mathematics},
	VOLUME = {129},
	YEAR = {2001},
	NUMBER = {1-2},
	PAGES = {15--35},
	ISSN = {0377-0427,1879-1778},
	MRCLASS = {90C53 (90C26)},
	MRNUMBER = {1823208},
	MRREVIEWER = {Hiroshi\ Yabe},
	DOI = {10.1016/S0377-0427(00)00540-9},
}

@article {Nesterov2005Siam,
	AUTHOR = {Nesterov, Yu.},
	TITLE = {Excessive gap technique in nonsmooth convex minimization},
	JOURNAL = {SIAM J. Optim.},
	FJOURNAL = {SIAM Journal on Optimization},
	VOLUME = {16},
	YEAR = {2005},
	NUMBER = {1},
	PAGES = {235--249},
	ISSN = {1052-6234,1095-7189},
	MRCLASS = {90C25 (49J52)},
	MRNUMBER = {2177777},
	MRREVIEWER = {V.\ F.\ Dem\cprime yanov},
	DOI = {10.1137/S1052623403422285},
}

@article {Perry1978,
	AUTHOR = {Perry, Avinoam},
	TITLE = {A modified conjugate gradient algorithm},
	JOURNAL = {Oper. Res.},
	FJOURNAL = {Operations Research},
	VOLUME = {26},
	YEAR = {1978},
	NUMBER = {6},
	PAGES = {1073--1078},
	ISSN = {0030-364X,1526-5463},
	MRCLASS = {65K05},
	MRNUMBER = {514875},
	DOI = {10.1287/opre.26.6.1073},
}

@article {Polak1969,
	AUTHOR = {Polak, E. and Ribi\`ere, G.},
	TITLE = {Note sur la convergence de m\'ethodes de directions
	conjugu\'ees},
	JOURNAL = {Rev. Fran\c caise Informat. Recherche Op\'erationnelle},
	FJOURNAL = {Revue Fran\c aise d'Informatique et de Recherche
	Op\'erationnelle},
	VOLUME = {3},
	YEAR = {1969},
	NUMBER = {16},
	PAGES = {35--43},
	ISSN = {0035-3035,2669-3224},
	MRCLASS = {65.30},
	MRNUMBER = {255025},
	MRREVIEWER = {P.\ J.\ Laurent},
}

@article {Raydon1997,
	AUTHOR = {Raydan, Marcos},
	TITLE = {The {B}arzilai and {B}orwein gradient method for the large
	scale unconstrained minimization problem},
	JOURNAL = {SIAM J. Optim.},
	FJOURNAL = {SIAM Journal on Optimization},
	VOLUME = {7},
	YEAR = {1997},
	NUMBER = {1},
	PAGES = {26--33},
	ISSN = {1052-6234},
	MRCLASS = {90C30 (65K05)},
	MRNUMBER = {1430555},
	MRREVIEWER = {Nada\ I.\ Djuranovi\'c-Mili\v ci\v c},
	DOI = {10.1137/S1052623494266365},
}

@article{Rangarajan,
  title={On the unification of line processes, outlier rejection, and robust statistics with applications in early vision},
  author={Black, Michael J and Rangarajan, Anand},
  journal={International journal of computer vision},
  volume={19},
  number={1},
  pages={57--91},
  year={1996},
  publisher={Springer}
}

@article {Shanno1978,
	AUTHOR = {Shanno, David F.},
	TITLE = {Conjugate gradient methods with inexact searches},
	JOURNAL = {Math. Oper. Res.},
	FJOURNAL = {Mathematics of Operations Research},
	VOLUME = {3},
	YEAR = {1978},
	NUMBER = {3},
	PAGES = {244--256},
	ISSN = {0364-765X,1526-5471},
	MRCLASS = {65K10 (90C30)},
	MRNUMBER = {506662},
	DOI = {10.1287/moor.3.3.244},
}

@article {Yu2008OMS,
	AUTHOR = {Yu, Gaohang and Guan, Lutai and Chen, Wufan},
	TITLE = {Spectral conjugate gradient methods with sufficient descent
	property for large-scale unconstrained optimization},
	JOURNAL = {Optim. Methods Softw.},
	FJOURNAL = {Optimization Methods \& Software},
	VOLUME = {23},
	YEAR = {2008},
	NUMBER = {2},
	PAGES = {275--293},
	ISSN = {1055-6788,1029-4937},
	MRCLASS = {90C30 (65K10 90C52)},
	MRNUMBER = {2403924},
	MRREVIEWER = {Ya\ Xiang\ Yuan},
	DOI = {10.1080/10556780701661344},
}

@article {Zhou2006,
	AUTHOR = {Zhou, Weijun and Zhang, Li},
	TITLE = {A nonlinear conjugate gradient method based on the {MBFGS}
	secant condition},
	JOURNAL = {Optim. Methods Softw.},
	FJOURNAL = {Optimization Methods \& Software},
	VOLUME = {21},
	YEAR = {2006},
	NUMBER = {5},
	PAGES = {707--714},
	ISSN = {1055-6788,1029-4937},
	MRCLASS = {90C52 (90C30)},
	MRNUMBER = {2238653},
	DOI = {10.1080/10556780500137041},
}

@article{Zoutendijk,
  title={Nonlinear programming, computational methods},
  author={Zoutendijk, G},
  journal={Integer and nonlinear programming},
  pages={37--86},
  year={1970},
  publisher={North-Holland}
}

@article {Zhu2013,
	AUTHOR = {Zhu, Hong and Xiao, Yunhai and Wu, Soon-Yi},
	TITLE = {Large sparse signal recovery by conjugate gradient algorithm
	based on smoothing technique},
	JOURNAL = {Comput. Math. Appl.},
	FJOURNAL = {Computers \& Mathematics with Applications. An International
	Journal},
	VOLUME = {66},
	YEAR = {2013},
	NUMBER = {1},
	PAGES = {24--32},
	ISSN = {0898-1221,1873-7668},
	MRCLASS = {65F22 (65K10 94A12)},
	MRNUMBER = {3063844},
	MRREVIEWER = {Stefan\ Sommer},
	DOI = {10.1016/j.camwa.2013.04.022},
}

@article {Mrad2024,
		AUTHOR = {Mrad, Hatem and Fakhari, Seyyed Mojtaba},
		TITLE = {Optimization of unconstrained problems using a developed
		algorithm of spectral conjugate gradient method calculation},
		JOURNAL = {Math. Comput. Simulation},
		FJOURNAL = {Mathematics and Computers in Simulation},
		VOLUME = {215},
		YEAR = {2024},
		PAGES = {282--290},
		ISSN = {0378-4754,1872-7166},
		MRCLASS = {99-01},
		MRNUMBER = {4632764},
		DOI = {10.1016/j.matcom.2023.07.026},
	}

@article {Kafaki2025,
		AUTHOR = {Babaie-Kafaki, Saman and Dargahi, Fatemeh and Aminifard,
		Zohre},
		TITLE = {On solving a revised model of the nonnegative matrix
		factorization problem by the modified adaptive versions of the
		{D}ai-{L}iao method},
		JOURNAL = {Numer. Algorithms},
		FJOURNAL = {Numerical Algorithms},
		VOLUME = {99},
		YEAR = {2025},
		NUMBER = {1},
		PAGES = {505--519},
		ISSN = {1017-1398,1572-9265},
		MRCLASS = {65K05 (15A23 15B48 90C53 90C90)},
		MRNUMBER = {4892880},
		DOI = {10.1007/s11075-024-01886-w},
	}

@article {Kafaki2025Opt,
		AUTHOR = {Babaie--Kafaki, Saman and Aminifard, Zohre},
		TITLE = {An augmented memoryless {D}avidon--{F}letcher--{P}owell method
		based on the {D}ai--{L}iao approach},
		JOURNAL = {Optimization},
		FJOURNAL = {Optimization. A Journal of Mathematical Programming and
		Operations Research},
		VOLUME = {74},
		YEAR = {2025},
		NUMBER = {15},
		PAGES = {4049--4061},
		ISSN = {0233-1934,1029-4945},
		MRCLASS = {99-06},
		MRNUMBER = {4979652},
		DOI = {10.1080/02331934.2024.2395432},
	}
			
			
			
			%
			%
			%
		\end{document}